\documentclass[a4paper]{article}
\usepackage{amssymb}
\usepackage{amsthm}
\usepackage{latexsym}
\usepackage{amsmath}
\usepackage{xcolor}

\usepackage{multirow}
\usepackage{comment}

\includecomment{note}

\theoremstyle{definition}
\newtheorem{dfn}{Definition}[section]
\newtheorem{lem}[dfn]{Lemma}
\newtheorem{thm}[dfn]{Theorem}

\newtheorem{rem}[dfn]{Remark}
\newtheorem{prop}[dfn]{Proposition}
\newtheorem{assumption}{Assumption}

\usepackage{mathtools}
\usepackage{stmaryrd}
\usepackage{float}
\usepackage{here}
\usepackage{mathrsfs}
\usepackage{enumitem}
\usepackage{bm}
\usepackage{bbm}
\usepackage{physics}
\usepackage{diagbox}
\usepackage[dvips]{graphicx}
\usepackage{hyperref}
\hypersetup{
    linktoc=page, 
    colorlinks=true,
    linkcolor=blue, 
    citecolor=blue
}

\newcommand{\MAT}[2]{\left(\begin{array}{#1} #2 \end{array} \right)} 

\makeatletter
\@addtoreset{equation}{section}
\makeatother

\begin{document}

\title{Asymptotically uniformly most powerful tests for diffusion processes with nonsynchronous observations}
\author{Teppei Ogihara$^*$ and Futo Ueno$^*$\\
$*$
\begin{small}Graduate School of Information Science and Technology, University of Tokyo, \end{small}\\
\begin{small}7-3-1 Hongo, Bunkyo-ku, Tokyo 113--8656, Japan \end{small}\\
}
\maketitle

\noindent
{\bf Abstract.}
This paper introduces a quasi-likelihood ratio testing procedure for diffusion processes observed under nonsynchronous sampling schemes. High-frequency data, particularly in financial econometrics, are often recorded at irregular time points, challenging conventional synchronous methods for parameter estimation and hypothesis testing. To address these challenges, we develop a quasi-likelihood framework that accommodates irregular sampling while integrating adaptive estimation techniques for both drift and diffusion coefficients, thereby enhancing optimization stability and reducing computational burden. We rigorously derive the asymptotic properties of the proposed test statistic, showing that it converges to a chi-squared distribution under the null hypothesis and exhibits consistency under alternatives. Moreover, we establish that the resulting tests are asymptotically uniformly most powerful. Extensive numerical experiments corroborate the theoretical findings and demonstrate that our method outperforms existing nonparametric approaches.

~


\section{Introduction}

In recent years, the analysis of diffusion processes using high-frequency data has gathered significant attention in various fields such as financial econometrics, engineering, and the natural sciences. In practical applications, however, the available data are often observed at nonsynchronous time points, where different components of the process are recorded at irregular and non-coinciding intervals. This nonsynchronous sampling scheme arises frequently, particularly when dealing with intra-day financial data, and poses substantial challenges to conventional statistical methods that assume synchronous observations, thereby affecting both parameter estimation and hypothesis testing procedures.

Historically, the development of asymptotically optimal tests was initiated by Wald \cite{wal43}, who demonstrated that maximum likelihood-based tests could effectively detect local alternatives. 
Choi et al.~\cite{cho-etal96} investigated the concept of asymptotically most powerful (AUMP) tests—optimal in terms of power against local alternatives—for hypothesis testing in statistical models involving nuisance parameters that satisfy local asymptotic normality. Furthermore, in the context of a multidimensional parameter space, they studied AUMPI tests, which incorporate the notion of rotational invariance in the limit of the tests.

Meanwhile, research on the estimation of parameters for diffusion processes based on discrete-time data has also been actively pursued.
Parametric estimation has been analyzed through quasi-maximum likelihood methods in Florens-Zmirou~\cite{flo89}, Yoshida~\cite{yos92, yos11}, Kessler~\cite{kes97}, and Uchida and Yoshida~\cite{uch-yos12}. In parallel, considerable research has focused on the estimation of quadratic covariation for diffusion processes observed nonsynchronous manner. 
In the case of nonsynchronous observations, it has been demonstrated that as the observation frequency increases, severe bias can arise in the estimated covariation when employing simple synchronization methods. Epps \cite{epp79} first identified this phenomenon through an analysis of US stock data, and it has come to be known as the Epps Effect. Furthermore, Hayashi and Yoshida \cite{hay-yos05} quantitatively analyzed this bias under Poisson sampling.
In this context, consistent estimators have been independently developed by Hayashi and Yoshida \cite{hay-yos05, hay-yos08, hay-yos11} as well as by Malliavin and Mancino \cite{mal-man02, mal-man09}.
Furthermore, Ogihara and Yoshida~\cite{ogi-yos14} and Ogihara~\cite{ogi23} constructed a quasi-log-likelihood function that accounts for nonsynchronous observations in parametric diffusion process models and derived maximum likelihood-type estimators for the parameters.

In nonsynchronous observation models of diffusion processes, hypothesis testing is as critical as parameter and covariance estimation. As an application in finance, for example, testing the covariance parameters of two asset price processes, or the residuals of their multi-factor models, can provide insights into their interdependence, which may then be utilized for risk management in asset allocation.
For diffusion processes with synchronous observations, hypothesis testing problems have been discussed by Gregorio and Iacus~\cite{gre-iac10, gre-iac13, gre-iac19} and Kitagawa and Uchida~\cite{kit-uch14}, among others; moreover, Nakakita and Uchida~\cite{nak-uch19} have investigated hypothesis testing for diffusion processes with noisy observations.
On the other hand, hypothesis testing for diffusion processes in nonsynchronous observation models remains almost unexplored.

In this paper, we propose quasi-likelihood ratio tests specifically designed for diffusion processes with nonsynchronous observations. Conducting hypothesis tests on the parameterized diffusion processes enables an assessment of the model structure’s validity. 
Our approach is built upon a quasi-likelihood framework that not only accommodates the irregular nature of the sampling scheme but also enhances the stability of the optimization process and reduces the computational cost through the adaptive estimation of the diffusion coefficient parameters and the drift coefficient parameters.
By deriving the asymptotic properties of the proposed test statistic, we show that it converges to a chi-squared distribution under the null hypothesis, and the test is consistent under alternatives. 
Furthermore, we extend the result of local asymptotic normality for the statistical model given by Ogihara~\cite{ogi23} to uniform local asymptotic normality. Together with previous results on asymptotic optimality of tests established by Hall and Mathiason~\cite{hal-mat90} and Choi, Hall, and Schick~\cite{cho-etal96}, we further demonstrate that the test statistics introduced in this study are AUMPI tests.
In numerical experiments, we verified that the proposed tests performance consistent with theoretical predictions under the null, alternative, and local alternative hypotheses, and we compared their performance with a test based on the nonparametric estimator in Hayashi and Yoshida~\cite{hay-yos05}.

Our contributions are threefold. First, we address a critical gap in the literature by developing a testing procedure that is robust to the challenges posed by nonsynchronous observations. Second, the integration of adaptive estimation techniques into the quasi-likelihood framework offers a fresh perspective on constructing tests that retain optimal asymptotic properties even in non-standard sampling settings. To the best of our knowledge, no study on test statistics for discretely observed diffusion process models has established the AUMPI property—even in the case of synchronous observations. Our results include the AUMPI property for synchronous observations as a special case.
Third, the comprehensive simulation analysis confirms that the proposed method outperforms existing tests, thereby providing a practical tool for researchers and practitioners working with high-frequency data.

The remainder of this paper is organized as follows. In Section \ref{main-section}, we describe the model settings and present the main results. Section \ref{sec:setting} outlines the model settings and assumptions, and introduces the quasi-likelihood functions necessary for constructing the test statistics. In Section \ref{consis-subsection}, we establish the consistency of the test statistics, while Section \ref{aump-subsection} discusses the AUMPI property of the proposed test. Section \ref{fast-calc-subsection} details the procedures for the constrained computation of the proposed test. In Section \ref{num-section}, we report numerical experiments, including tests on the diffusion coefficient parameters under both time-independent and time-dependent scenarios, as well as experiments on a drift coefficient parameter test. Finally, Section \ref{proof-section} provides the proofs of the main results.

\subsection*{Notation}

Throughout this study, we often regard a $p$-dimensional vector as a $p \times 1$ matrix. For a given set $S$, its closure is denoted by $\overline{S}$. For a matrix $A$, $A^{\top}$ denotes the transpose of a matrix $A$, and \( \|A\| \) denotes the operator norm.
The identity matrix of dimension $d$ is represented by $\mathcal{E}_{d}$. For a $k \times l$ matrix $A$, the $(i,j)$-th element is denoted by $[A]_{ij}$. For a vector $V = (v_i)_{i=1}^{k}$, ${\rm diag}(V)$ represents the $k \times k$ diagonal matrix with diagonal elements $[{\rm diag}(V)]_{ii} = v_i$. For a vector $x = (x_1, \dots, x_k)$ and a positive integer $l$, we define the differential operator as $\partial_x^l = (\frac{\partial^l}{\partial x_{i_1} \dots \partial x_{i_l}})_{i_1,\dots,i_l=1}^{k}$, where $\partial_x^0$ denotes the identity operator. The $l$-th component of a vector $x$ is denoted by $x^{(l)}$.
For a collection \( \{x_{i_{1},\ldots ,i_{D}}\}_{i_{1},\ldots ,i_{D}} \), we set  
\[
\lvert x \rvert = \sqrt{\sum_{i_{1},\ldots ,i_{D}} x_{i_{1},\ldots ,i_{D}}^{2}}.
\]

\section{Main results}\label{main-section}
\subsection{Settings and the quasi-likelihood ratio test statistics}\label{sec:setting}

Let $(\Omega, \mathcal{F}, P)$ be a probability space, and let $\mathbb{F} = \{\mathcal{F}_{t}\}_{t \geq 0}$ be a right-continuous filtration. We consider the following stochastic differential equation (SDE) parameterized by $\sigma$ and $\theta$:
\begin{equation}
    \label{SDE}
    \dd X_{t}^{(\gamma)} = \mu_{t}(\theta)\dd t + b_{t}(\sigma)\dd W_{t}, \quad X_{0}^{(\gamma)} = x_{0},
\end{equation}
which describes a two-dimensional diffusion process $\{X_{t}^{(\gamma)}\}_{t\geq 0} = \{(X_{t}^{(\gamma), 1}, X_{t}^{(\gamma), 2})^{\top}\}_{t\geq 0}$. Here, $\{W_{t}\}_{t\geq 0}$ is a two-dimensional standard $\mathbb{F}$-Wiener process, $\mu_{t}(\theta)$ and $b_{t}(\sigma)$ are $\mathbb{R}^{2}$- and $\mathbb{R}^{2} \times \mathbb{R}^{2}$-valued deterministic functions, respectively. The parameter vector is given by $\gamma = (\sigma^{\top}, \theta^{\top})^{\top}$, where $x_{0} \in \mathbb{R}^{2}$, $\sigma \in \Theta_{1} \subset \mathbb{R}^{d_{1}}$, and $\theta \in \Theta_{2} \subset \mathbb{R}^{d_{2}}$. The parameter spaces $\Theta_{1}$ and $\Theta_{2}$ are open sets containing zero, and satisfies the following Sobolev inequality: for any $l=1,2$, $u\in C^{1}(\Theta_l)$ and $q>d_l$,
\begin{eqnarray}
\sup_{x\in\Theta_l}\lvert u(x) \rvert \le C_q\left\{\left(\int_{\Theta_l}\lvert u(x) \rvert^q \dd x\right)^{1/q}
+\left(\int_{\Theta_l} \lvert \partial_xu(x) \rvert^q \dd x\right)^{1/q}\right\}.
\end{eqnarray}
Here, $C_q$ is a positive constant depending only on $q$.
The Sobolev inequality can be derived if $\Theta_l$ has a Lipschitz boundary.
See Adams and Founier~\cite{ada-fou03} for the detail.

The true parameter value is denoted by $\gamma_{0} = (\sigma_{0}^{\top}, \theta_{0}^{\top})^{\top} \in \Theta_{1} \times \Theta_{2}$. For simplicity, we write $X_{t} = X_{t}^{(\gamma_{0})} = (X_{t}^{1}, X_{t}^{2})^{\top}$. Given a positive integer $n$ and some $\epsilon_0 > 0$, we assume a positive sequence $(h_{n})_{n=1}^{\infty}$ satisfying
\[
    h_{n} \to 0, \quad n^{1-\epsilon_{0}} h_{n} \to \infty, \quad nh_{n}^{2} \to 0
\]
as $n\to\infty$.
Setting $T_{n} = nh_{n}$, we consider a setting where the process $X_t$ is observed nonsynchronously and randomly at discrete times. For $l \in \{1, 2\}$, the observation times of $X_{t}^{l}$ are represented by a strictly increasing sequence of random variables $\{S_{i}^{n, l}\}_{i=0}^{M_{l}}$ satisfying $S_{0}^{n, l} = 0$ and $S_{M_{l}}^{n, l} = T_{n}$. Here, $M_{l}$ is a positive integer-valued random variable that depends on $n$. Intuitively, $n$ represents the order of the number of observations, $h_{n}$ represents the order of the observation interval, and $M_{l}$ represents the number of observations of $X_{t}^{l}$.

For $1 \leq r_{1} \leq d_{1}$, we formulate the hypothesis testing problem for the parameter $\sigma$ as follows:
\begin{align*}
    H_0^{(\sigma)}&: \sigma^{(1)} = \cdots = \sigma^{(r_{1})} = 0, \\
    H_1^{(\sigma)}&: \text{not } H_0^{(\sigma)}.
\end{align*}
Similarly, for $1 \leq r_{2} \leq d_{2}$, we define the hypothesis testing problem for the parameter $\theta$ as follows:
\begin{align*}
    H_0^{(\theta)}&: \theta^{(1)} = \cdots = \theta^{(r_{2})} = 0, \\
    H_1^{(\theta)}&: \text{not } H_0^{(\theta)}.
\end{align*}
In this study, we aim to introduce test statistics that achieve consistency for nonsynchronous observations of diffusion processes.

For \(1 \leq i \leq M_{l}\), we define \( I_{i}^{l} = (S_{i-1}^{n, l} , S_{i}^{n, l}] \). The increment of \( X^{l}_{t} \) over the interval \( I_{i}^{l} \) is given by \( \Delta_{i}^{l} X = X^{l}_{S_{i}^{n, l}} - X^{l}_{S_{i-1}^{n, l}} \). We then define \( \Delta^{l} X = (\Delta_{i}^{l} X)_{1 \leq i \leq M_{l}} \) and  
\[
\Delta X = ((\Delta^{1}X)^{\top}, (\Delta^{2}X)^{\top})^{\top}.
\]
In other words, \( \Delta X \) is a concatenated vector consisting of the increments of \( X^{1}_{t} \) and \( X^{2}_{t} \), arranged in vector form.

Next, we define
\[
\tilde{\Sigma}_{i}^{l}(\sigma) = \int_{I_{i}^{l}}[b_{t}b_{t}^{\top}(\sigma)]_{ll}\dd t, \quad
\tilde{\Sigma}_{i, j}^{1, 2}(\sigma) = \int_{I_{i}^{1}\cap I_{j}^{2}}[b_{t}b_{t}^{\top}(\sigma)]_{12}\dd t,
\]
and construct the matrix
\[
    S_{n}(\sigma) =
    \begin{pmatrix}
    \mathrm{diag}\left((\tilde{\Sigma}_{i}^{1})_{1\leq i \leq M_{1}}\right) & \left(\tilde{\Sigma}_{i,j}^{1,2}\right)_{\substack{1 \leq i \leq M_{1}, \\1 \leq j \leq M_{2}}} \\ \\
    \left(\tilde{\Sigma}_{i,j}^{1,2}\right)^{\top}_{\substack{1 \leq i \leq M_{1}, \\1 \leq j \leq M_{2}}} & \mathrm{diag}\left((\tilde{\Sigma}_{j}^{2})_{1\leq j \leq M_{2}}\right)
    \end{pmatrix}.
\]
Using this, we define the quasi-log-likelihood function \( \mathbb{H}_{n}^{1}(\sigma) \) with respect to \( \sigma \) as follows:
\[
    \mathbb{H}_{n}^{1}(\sigma) = -\frac{1}{2}\Delta X^{\top}S_{n}^{-1}(\sigma)\Delta X - \frac{1}{2}\log\det S_{n}(\sigma).
\]
Then, the maximum likelihood-type estimator of \( \sigma \) is given by
\[
    \hat{\sigma}_{n} \in \underset{\sigma\in\overline{\Theta}_{1}} {\operatorname{argmax}}\ \mathbb{H}_{n}^{1}(\sigma).
\]

Furthermore, we define
\[
\Delta^{l} V(\theta) = \bigg(\int_{I_{i}^{l}}\mu_{t}^{l}(\theta)\dd t\bigg)_{1 \leq i \leq M_{l}}, \quad
\Delta V(\theta) = (\Delta^{1} V(\theta)^{\top}, \Delta^{2} V(\theta)^{\top})^{\top}
\]
and $\bar{X}(\theta) = \Delta X - \Delta V(\theta)$, and introduce the quasi-log-likelihood function \( \mathbb{H}_{n}^{2}(\theta) \) with respect to \( \theta \) as
\[
    \mathbb{H}_{n}^{2}(\theta) = -\frac{1}{2}\bar{X}(\theta)^{\top}S_{n}^{-1}(\hat{\sigma}_{n})\bar{X}(\theta).
\]
Thus, the maximum likelihood-type estimator of \( \theta \) is given by
\[
    \hat{\theta}_{n} \in \underset{\theta\in\overline{\Theta}_{2}} {\operatorname{argmax}}\ \mathbb{H}_{n}^{2}(\theta).
\]
The estimators \( \hat{\sigma}_n \) and \( \hat{\theta}_n \), along with the quasi-log-likelihood functions \( \mathbb{H}_{n}^{1}(\sigma) \) and \( \mathbb{H}_{n}^{2}(\theta) \), were introduced in Ogihara~\cite{ogi23} as maximum likelihood-type estimators and quasi-log-likelihood functions.

Then, we define the proposed test statistics as follows.
Define
\[
\Theta_{0, 1} = \{\sigma \in \Theta_{1} \mid \sigma^{(1)} = \cdots = \sigma^{(r_1)} = 0\},
\]
that is, \(\Theta_{0, 1}\) is the restricted parameter space under \(H_{0}^{(\sigma)}\). Moreover, let
\[
\tilde{\sigma}_{n} \in \underset{\sigma \in \overline{\Theta}_{0,1}}{\operatorname{argmax}}\, \mathbb{H}_{n}^{1}(\sigma),
\]
so that \(\tilde{\sigma}_{n}\) is the maximum likelihood-type estimator under \(H_{0}^{(\sigma)}\). The quasi-likelihood ratio test statistic \(T_{n}^{1}\) is then defined by
\begin{align*}
    T_{n}^{1} &= 2\Bigl(\mathbb{H}_{n}^{1}(\hat{\sigma}_{n}) - \mathbb{H}_{n}^{1}(\tilde{\sigma}_{n})\Bigr) \\
    &= -2\left(\log\frac{\max_{\sigma \in \overline{\Theta}_{0,1}} \exp\bigl\{\mathbb{H}_{n}^{1}(\sigma)\bigr\}}{\max_{\sigma \in \overline{\Theta}_{1}} \exp\bigl\{\mathbb{H}_{n}^{1}(\sigma)\bigr\}}\right).
\end{align*}

Similarly, define
\[
\Theta_{0, 2} = \{\theta \in \Theta_{2} \mid \theta^{(1)} = \cdots = \theta^{(r_2)} = 0\},
\]
and, in analogy with \(\tilde{\sigma}_{n}\), define the maximum likelihood-type estimator
\[
\tilde{\theta}_{n} \in \underset{\theta \in \overline{\Theta}_{0,2}}{\operatorname{argmax}}\, \mathbb{H}_{n}^{2}(\theta).
\]
The quasi-likelihood ratio test statistic \(T_{n}^{2}\) is then defined by
\[
T_{n}^{2} = 2\Bigl(\mathbb{H}_{n}^{2}(\hat{\theta}_{n}) - \mathbb{H}_{n}^{2}(\tilde{\theta}_{n})\Bigr)
= -2\left(\log\frac{\max_{\theta \in \overline{\Theta}_{0,2}} \exp\bigl\{\mathbb{H}_{n}^{2}(\theta)\bigr\}}{\max_{\theta \in \overline{\Theta}_{2}} \exp\bigl\{\mathbb{H}_{n}^{2}(\theta)\bigr\}}\right).
\]

\subsection{Consistency of the quasi-likelihood ratio tests}\label{consis-subsection}

To state the main theorems, we impose the following assumptions (A1) -- (A6).  
These conditions were introduced in Ogihara~\cite{ogi23} to derive the asymptotic properties of \( \mathbb{H}_n^1 \), \( \mathbb{H}_n^2 \), \( \hat{\sigma}_n \), and \( \hat{\theta}_n \).  

First, we define \( \Sigma_{t}(\sigma) = b_{t}b_{t}^{\top}(\sigma) \).

\begin{assumption}
There exists a constant \( c_{1} > 0 \) such that for all \( t \in [0,\infty) \) and \( \sigma \in \Theta_{1} \),  
\[
c_{1} \mathcal{E}_{2} \leq \Sigma_{t}(\sigma).
\]
For \( k \in \{0, 1, 2, 3, 4\} \), the derivatives \( \partial_{\theta}^{k}\mu_{t}(\theta) \) and \( \partial_{\sigma}^{k}b_{t}(\sigma) \) exist and are continuous in \( (t, \sigma, \theta) \in [0, \infty) \times \overline{\Theta}_{1} \times \overline{\Theta}_{2} \).  
Furthermore, for any \( \epsilon > 0 \), there exist constants \( \delta > 0 \) and \( K > 0 \) such that for all \( k \in \{0, 1, 2, 3, 4\} \), \( \sigma \in \Theta_{1} \), \( \theta \in \Theta_{2} \), and \( t, s \geq 0 \) satisfying \( \lvert t - s \rvert < \delta \),  
\[
\lvert \partial_{\theta}^{k}\mu_{t}(\theta) \rvert + \lvert \partial_{\sigma}^{k}b_{t}(\sigma) \rvert \leq K,  
\]
\[
\lvert \partial_{\theta}^{k}\mu_{t}(\theta) - \partial_{\theta}^{k}\mu_{s}(\theta) \rvert + \lvert \partial_{\sigma}^{k}b_{t}(\sigma) - \partial_{\sigma}^{k}b_{s}(\sigma) \rvert \leq \epsilon.
\]

\end{assumption}

Let \( r_{n} = \max_{i,l} \lvert I_{i}^{l} \rvert \).
\begin{assumption}
$r_{n} \overset{P}{\rightarrow} 0$ as $n \to \infty$.
\end{assumption}

We introduce the following notations:  
\[
    \rho_{t}(\sigma) = \frac{[\Sigma_{t}]_{12}}{[\Sigma_{t}]_{11}^{1/2}[\Sigma_{t}]_{22}^{1/2}}(\sigma), \quad \rho_{t,0} = \rho_{t}(\sigma_{0}),
\]
\[
    B_{l,t}(\sigma) = \frac{[\Sigma_{t}(\sigma_{0})]_{ll}}{[\Sigma_{t}(\sigma)]_{ll}}, \quad \phi_{l,t}(\theta) = [\Sigma_{t}(\sigma_{0})]_{ll}^{-1/2}(\mu_{t}^{l}(\theta) - \mu_{t}^{l}(\theta_{0})).
\]

\begin{assumption}
For any \( l \in \{1, 2\} \), \( i_{1}, i_{2} \in \{0, 1\} \), \( i_{3} \in \{0, 1, 2, 3, 4\} \), and \( k_{1}, k_{2} \in \{0, 1, 2\} \) satisfying \( k_{1} + k_{2} = 2 \),  
there exist continuous functions \( \Phi_{i_{1}, i_{2}}^{1, F}(\sigma) \), \( \Phi_{l, i_{3}}^{2}(\sigma) \) (defined on \( \overline{\Theta}_{1} \)), and \( \Phi_{i_{1}, i_{3}}^{3, k_{1}, k_{2}}(\theta) \) (defined on \( \overline{\Theta}_{2} \)) such that for any polynomial \( F(x_{1}, \ldots, x_{14}) \) of degree at most 6,  
\[
    \lim_{T\to\infty} \frac{1}{T} \int_{0}^{T} F\left((\partial_{\sigma}^{k} B_{l,t}(\sigma))_{0 \leq k \leq 4, l = 1,2}, (\partial_{\sigma}^{k'} \rho_{t}(\sigma))_{1 \leq k' \leq 4}\right) \rho_{t}(\sigma)^{i_{1}} \rho_{t,0}^{i_{2}} \dd t = \Phi_{i_{1}, i_{2}}^{1, F}(\sigma),
\]
\[
    \lim_{T\to\infty} \frac{1}{T} \int_{0}^{T} \partial_{\sigma}^{i_{3}} \log B_{l,t}(\sigma) \dd t = \Phi_{l, i_{3}}^{2}(\sigma),
\]
\[
    \lim_{T\to\infty} \frac{1}{T} \int_{0}^{T} \partial_{\theta}^{i_{3}} (\phi_{1,t}^{k_{1}} \phi_{2,t}^{k_{2}})(\theta) \rho_{t,0}^{i_{1}} \dd t = \Phi_{i_{1}, i_{3}}^{3, k_{1}, k_{2}}(\theta).
\]
\end{assumption}

Let \( \mathfrak{S} \) denote the set of all partitions \( (s_k)_{k=0}^{\infty} \) of \( [0, \infty) \) satisfying  
\[
    \sup_{k\geq 1} \lvert s_{k} - s_{k-1} \rvert \leq 1, \quad \inf_{k\geq 1} \lvert s_{k} - s_{k-1} \rvert > 0.
\]
For \( (s_k)_{k=0}^{\infty} \in \mathfrak{S} \), we define  
\[
    M_{l,k} = \#\{i;\ \sup I_{i}^{l} \in (s_{k-1}, s_{k}]\},
\]
\[
    q_{n} = \max\{k;\ s_{k} \leq nh_{n} \}.
\]
We also define the matrix \( \mathcal{E}_{(k)}^{l} \) by  
\[
    [\mathcal{E}_{(k)}^{l}]_{ij} =
    \begin{cases}
        1, & \text{if } i = j \text{ and } \sup I_{i}^{l} \in (s_{k-1}, s_{k}], \\
        0, & \text{otherwise}.
    \end{cases}
\]
Furthermore, we define  
\[
    G = 
    \left\{
        \frac{\lvert I_{i}^{1} \cap I_{j}^{2} \rvert}{\lvert I_{i}^{1}\rvert^{1/2} \lvert I_{j}^{2}\rvert^{1/2}}
    \right\}_{1 \leq i \leq M_{1}, 1 \leq j \leq M_{2}}.
\]

\begin{assumption}
\label{A4}
For \( l \in \{1,2\} \), the sequence \( \{h_{n} M_{l,q_{n}+1}\}_{n=1}^{\infty} \) is \( P \)-tight,
and there exist positive constants \( a_{0}^{1}, a_{0}^{2} \) such that for any partition \( (s_k)_{k=0}^{\infty} \in \mathfrak{S} \),  
        \[
            \max_{1 \leq k \leq q_{n}} \lvert h_{n} M_{l,k} - a_{0}^{l} (s_{k} - s_{k-1}) \rvert \overset{P}{\rightarrow} 0.
        \]
Moreover, for any positive integer \( p \), there exists a nonnegative constant \( a_{p}^{1} \) such that for any partition \( (s_k)_{k=0}^{\infty} \in \mathfrak{S} \),  
        \[
            \max_{1 \leq k \leq q_{n}} \lvert h_{n} \tr(\mathcal{E}_{(k)}^{1} (GG^{\top})^{p}) - a_{p}^{1} (s_{k} - s_{k-1}) \rvert \overset{P}{\rightarrow} 0.
        \]
\end{assumption}

Let \( \mathfrak{I}_{l} = (\lvert I_{i}^{l} \rvert^{1/2})_{i=1}^{M_{l}} \).
\begin{assumption}
For \( l \in \{1,2\} \), the sequence \( \{\lvert \mathcal{E}_{(q_{n}+1)}^{l} \mathfrak{I}_{l} \rvert\}_{n=1}^{\infty} \) is \( P \)-tight,
and for any nonnegative integer \( p \), there exist nonnegative constants \( f_{p}^{1,1}, f_{p}^{1,2}, f_{p}^{2,2} \) such that for any partition \( (s_k)_{k=0}^{\infty} \in \mathfrak{S} \),  
        \[
            \max_{1 \leq k \leq q_{n}} \lvert\mathfrak{I}_{1}^{\top} \mathcal{E}_{(k)}^{1} (GG^{\top})^{p} \mathfrak{I}_{1} - f_{p}^{1,1} (s_{k} - s_{k-1})\rvert \overset{P}{\rightarrow} 0,
        \]
        \[
            \max_{1 \leq k \leq q_{n}} \lvert\mathfrak{I}_{1}^{\top} \mathcal{E}_{(k)}^{1} (GG^{\top})^{p} G \mathfrak{I}_{2} - f_{p}^{1,2} (s_{k} - s_{k-1})\rvert \overset{P}{\rightarrow} 0,
        \] 
        \[
            \max_{1 \leq k \leq q_{n}} \lvert\mathfrak{I}_{2}^{\top} \mathcal{E}_{(k)}^{2} (G^{\top} G)^{p} \mathfrak{I}_{2} - f_{p}^{2,2} (s_{k} - s_{k-1})\rvert \overset{P}{\rightarrow} 0.
        \]
\end{assumption}

Assumptions (A3), (A4), and (A5) are primarily required to characterize the asymptotic behavior of  
$n^{-1} (\mathbb{H}_{n}^{1}(\sigma) - \mathbb{H}_{n}^{1}(\sigma_{0}))$
and  
$(nh_{n})^{-1} (\mathbb{H}_{n}^{2}(\theta) - \mathbb{H}_{n}^{2}(\theta_{0}))$
as well as their derivatives.  
For further details, see Proposition 3.8 and Proposition 3.16 in Ogihara~\cite{ogi23}.
Assumption (A3) is satisfied if $\mu_t(\theta)$ and $b_t(\sigma)$ are independent of $t$, or are periodic functions with respect to $t$ having a common period (when the period does not depend on $\sigma$ nor $\theta$).
Sufficient conditions for Assumptions (A4) and (A5) are studied in Section 2.4 of Ogihara~\cite{ogi23}.

\begin{assumption}
\( a_{1}^{1} > 0 \),
and there exist positive constants \( c_{2} \) and \( c_{3} \) such that for any \( \sigma \in \overline{\Theta}_1 \) and \( \theta \in \overline{\Theta}_2 \),  
        \[
            \limsup_{T\to\infty} \left( \frac{1}{T} \int_{0}^{T} \|\Sigma_{t}(\sigma) - \Sigma_{t}(\sigma_{0})\|^{2} \dd t \right) \geq c_{2} \lvert\sigma - \sigma_{0}\rvert^{2},
        \]
        \[
            \limsup_{T\to\infty} \left( \frac{1}{T} \int_{0}^{T} \lvert\mu_{t}(\theta) - \mu_{t}(\theta_{0})\rvert^{2} \dd t \right) \geq c_{3} \lvert\theta - \theta_{0}\rvert^{2}.
        \]
        
\end{assumption}

Assumption (A6) is {\it an identifiability condition} for the parameters \( \sigma \) and \( \theta \).

We define \( \partial_{\sigma}B_{l, t, 0} = \partial_{\sigma}B_{l, t}(\sigma_{0}) \) and  
\[
    \mathcal{A}(\rho) = \sum_{p=1}^{\infty} a_{p}^{1} \rho^{2p}, \quad \rho \in (-1,1).
\]
Furthermore, we introduce  
\begin{align*}
    \gamma_{1,t} &= \mathcal{A}(\rho_{t, 0})\left(\frac{\partial\rho_{t, 0}}{\rho_{t, 0}} - \partial_{\sigma}B_{1,t,0} - \partial_{\sigma}B_{2,t,0}\right)^2 - \partial_{\rho}\mathcal{A}(\rho_{t, 0})\frac{(\partial\rho_{t, 0})^{2}}{\rho_{t, 0}}
    \\ 
    & \quad -2\sum_{l=1}^{2}(a_{0}^{l} + \mathcal{A}(\rho_{t, 0}))(\partial_{\sigma}B_{l,t,0})^{2},
\end{align*}
and define  
\[
    \Gamma_{1} = \lim_{T\to\infty}\frac{1}{T}\int_{0}^{T} \gamma_{1,t} \dd t.
\]
According to Proposition 3.15 of Ogihara~\cite{ogi23}, under Assumptions (A1) -- (A4) and (A6), the matrix \( \Gamma_{1} \) exists and is a symmetric, positive definite matrix.

Similarly, we regard \( \partial_{\theta}\phi_{l, t, 0} = \partial_{\theta}\phi_{l, t}(\theta_{0}) \) and  
\[
    \gamma_{2,t} = \sum_{p=0}^{\infty} \rho_{t, 0}^{2p} \left\{\sum_{l=1}^{2} f_{p}^{l, l} (\partial_{\theta}\phi_{l, t, 0})^{2} - 2\rho_{t, 0} f_{p}^{1, 2} \partial_{\theta}\phi_{1, t, 0} \partial_{\theta}\phi_{2, t, 0} \right\}.
\]
Then, we define  
\[
    \Gamma_{2} = \lim_{T\to\infty}\frac{1}{T}\int_{0}^{T} \gamma_{2,t} \dd t.
\]
According to the proof of Theorem 2.3 in Ogihara~\cite{ogi23}, under Assumptions (A1) -- (A6), the matrix \( \Gamma_{2} \) exists and is a positive definite symmetric matrix.
The matrices \( \Gamma_{1} \) and \( \Gamma_{2} \) 
characterize the asymptotic variances of the maximum likelihood-type estimators of \( \sigma \) and \( \theta \).

Now we state our main results.

\begin{thm}\label{thm:sigma0}
Suppose that Assumptions (A1)--(A4) and (A6) hold. Under \(H_{0}^{(\sigma)}\),
\[
    T_{n}^{1} \overset{d}{\rightarrow} \chi^2_{r_{1}} \quad {\rm as} \quad n \to \infty.
\]
\end{thm}

\begin{thm}\label{thm:sigma1}
Suppose that Assumptions (A1)--(A4) and (A6) hold. Under \(H_{1}^{(\sigma)}\), for any \(M > 0\),
\[
    P\bigl(T_{n}^{1} \leq M\bigr) \to 0 \quad {\rm as} \quad n \to \infty.
\]
\end{thm}

\begin{thm}\label{thm:theta0}
Suppose that Assumptions (A1)--(A6) hold. Under \(H_{0}^{(\theta)}\),
\[
    T_{n}^{2} \overset{d}{\rightarrow} \chi^2_{r_{2}} \quad {\rm as} \quad n \to \infty.
\]
\end{thm}

\begin{thm}\label{thm:theta1}
Suppose that Assumptions (A1)--(A6) hold. Under \(H_{1}^{(\theta)}\), for any \(M > 0\),
\[
    P\bigl(T_{n}^{2} \leq M\bigr) \to 0 \quad {\rm as} \quad n \to \infty.
\]
\end{thm}

These theorems imply that, in the asymptotic regime where the sample size is sufficiently large, consistent tests can be conducted by employing \(T_n^1\) and \(T_n^2\) in conjunction with the chi-squared distribution.
For large samples, the proposed test statistics not only maintain strict control of the Type I error rate at the predetermined significance level but also exhibit high power in detecting deviations from the null hypothesis. 

\subsection{The AUMPI property of the test statistics}\label{aump-subsection}

In this section, we introduce the concept of AUMPI as defined by Choi et al.~\cite{cho-etal96}, and we present the results demonstrating that the tests constructed from \(T_n^1\) and \(T_n^2\) are AUMPI.

Consider a general parameter space \(\Theta \subset \mathbb{R}^d\) and, for \(1 \leq r \leq d\), let the parameter \(\gamma = (\vartheta, \eta) \in \Theta\) be such that \(\vartheta \in \mathbb{R}^r\) and \(\eta \in \mathbb{R}^{d-r}\). Here, \(\eta\) represents a nuisance parameter; in the case where \(r=d\), we ignore $\eta$. For a fixed \(\vartheta_0 \in \mathbb{R}^r\), we consider the hypothesis test
\[
H_0: \vartheta = \vartheta_0 \quad \text{versus} \quad H_1: \vartheta \neq \vartheta_0.
\]
Assume that \((\vartheta_0,\eta) \in \Theta\) and let \(\mathscr{H}_\vartheta \subset \mathbb{R}^r\) and \(\mathscr{H}_\eta \subset \mathbb{R}^{d-r}\) be subsets that contain a neighborhood of 0. Moreover, let \(\{\delta_n\}_{n=1}^\infty\) be a sequence of positive numbers satisfying \(\delta_n \to 0\) as \(n \to \infty\), and define
\[
\gamma_n(h)=\gamma_n(h_\vartheta, h_\eta) = \bigl(\vartheta_0 + \delta_n h_\vartheta,\, \eta + \delta_n h_\eta\bigr),
\]
for $h=(h_\vartheta,h_\eta)\in \mathscr{H}_\vartheta \times \mathscr{H}_\eta$.
Then, for \(\alpha\in (0,1)\), we say that a sequence of tests \(\{\psi_n\}_{n=1}^\infty\) is \emph{asymptotically of level \(\alpha\) at \(\eta\)} if, for every \(h_\eta\in \mathscr{H}_\eta\),
\[
\limsup_{n\to\infty} E_{\gamma_n(0,h_\eta)}[\psi_n] \le \alpha,
\]
where $E_\gamma$ denotes the expectation with respect to $P_{\gamma,n}$.

Suppose that the family 
$\{P_{\gamma,n}\}_{\gamma \in \Theta,\, n\geq 1}$
 of probability measures satisfies local asymptotic normality at \((\vartheta_0,\eta)\); that is, for some symmetric, positive definite matrix $B$ and some random vector $S_n$,
\[
\log \frac{\dd P_{\gamma_n(h), n}}{\dd P_{(\vartheta_0,\eta), n}} - \Bigl( h^\top S_n - \frac{1}{2} h^\top B h \Bigr) \to 0,
\]
in \(P_{(\vartheta_0,\eta), n}\)-probability, for any $h\in \mathscr{H}_\vartheta \times \mathscr{H}_\eta$, and 
$S_n \stackrel{d}{\to} \mathcal{N}(0,B)$ under \(P_{(\vartheta_0,\eta), n}\). 
Decompose the matrix \(B\) into the submatrices \(B_{11}\), \(B_{12}\), and \(B_{22}\) of dimensions \(r\times r\), \(r\times (d-r)\), and \((d-r)\times (d-r)\), respectively, so that
\[
B = \MAT{cc}{B_{11} & B_{12} \\ B_{12}^\top & B_{22}}.
\]

In this setting, for any sequence of tests \(\{\psi_n\}_{n=1}^\infty\) and any \(h_\eta^0 \in \mathscr{H}_\eta\), Lemma 1 in the appendix of Choi et al.~\cite{cho-etal96} implies the existence of a limit test \(\varphi:\mathbb{R}^r\to [0,1]\) such that, for any subsequence \(\{n'\}\), we can find a further subsequence \(\{n''\}\) thereof, we have
\[
\lim_{n''\to\infty} E_{\gamma_{n''(h)}}[\psi_{n''}] = \int \varphi(z) \, \dd\Phi_r\Bigl(z - \Bigl(B_{11} - B_{12}B_{22}^{-1}B_{12}^\top\Bigr)h_\vartheta\Bigr),
\]
for any \(h=(h_\vartheta,\, h_\eta^0 - B_{22}^{-1}B_{12}^\top h_\vartheta)\) with $h_\vartheta \in \mathscr{H}_\vartheta$, where \(\Phi_r\) denotes the \(r\)-dimensional standard normal distribution.
In this context, if \(\varphi\) satisfies
$\varphi(u) = \varphi(R^\top u)$
for every \(r\times r\) orthogonal matrix \(R\) and every \(u\in \mathbb{R}^r\), then the sequence of tests \(\{\psi_n\}_{n=1}^\infty\) is said to be \emph{asymptotically invariant at \(\eta\)}.
As discussed in Choi et al.~\cite{cho-etal96}, 
such invariance is essential for establishing the asymptotic optimality of power in a multidimensional parameter space because, without it, optimal tests cannot exist.

The sequence of tests \(\{\psi_n\}_{n=1}^\infty\) is said to be \emph{asymptotically uniformly most powerful among all tests that are asymptotically invariant at \(\eta\) and of asymptotic level \(\alpha\) at \(\eta\)} (abbreviated as AUMPI\((\alpha,\eta)\)) if, for every \(h_\eta\in \mathscr{H}_\eta\) and every \(h_\vartheta\in \mathscr{H}_\vartheta \setminus \{0\}\), and for every sequence of tests \(\{\psi'_n\}_{n=1}^\infty\) that is asymptotically invariant at \(\eta\) and of asymptotic level \(\alpha\) at \(\eta\), it holds that
\[
\liminf_{n\to\infty} E_{\gamma_n((h_\vartheta,h_\eta))}[\psi_n] \ge \limsup_{n\to\infty} E_{\gamma_n((h_\vartheta,h_\eta))}[\psi'_n].
\]

\begin{rem}
    Choi et al.~\cite{cho-etal96} considered only the case where \(\delta_n = n^{-1/2}\); however, by carefully checking their proofs, their theory can be extended to the more general case where \(\delta_n \to 0\).
\end{rem}
~

Returning to our setting of the diffusion process model with nonsynchronous observations, we consider the scenario in which only one of the parameters—either \(\sigma\) or \(\theta\)—is perturbed under a local alternative.
For $\alpha\in (0,1)$ and $l\in \{1,2\}$, let
$$\psi_n^l=1_{\{T_n^l\geq \chi_{l}^2(\alpha)\}},$$
where \(\chi_{l}^{2}(\alpha)\) denotes the $1-\alpha$ quantile of \(\chi_{r_l}^{2}\).
Then by Theorems~\ref{thm:sigma0} and~\ref{thm:theta0},
$\psi_n^1$ and $\psi_n^2$ are asymptotically of level $\alpha$.
 \begin{thm}\label{aumpi-thm}
 Assume (A1)--(A6). Then, 
 \begin{enumerate}
     \item $\{\psi_n^1\}_{n=1}^\infty$ is \(\mathrm{AUMPI}(\alpha,\eta)\) with $\vartheta_0 = (\sigma_0^{(1)},\cdots, \sigma_0^{(r_1)})$, $\eta = (\sigma_0^{(r_1+1)},\cdots, \sigma_0^{(d_1)})$, and $\delta_n=n^{-1/2}$,
\item $\{\psi_n^2\}_{n=1}^\infty$ is \(\mathrm{AUMPI}(\alpha,\eta)\) with $\vartheta_0=(\theta_0^{(1)},\cdots, \theta_0^{(r_2)})$, $\eta=(\theta_0^{(r_2+1)},\cdots, \theta_0^{(d_2)})$, and $\delta_n=(nh_n)^{-1/2}$. 
 \end{enumerate}
\end{thm}

\subsection{Fast computation of \(\mathbb{H}_n^l\)}\label{fast-calc-subsection}

\subsubsection{Size reduction of $S_n$ by dividing observations into blocks}\label{divide-block-subsubsection}

The computation of the quasi-likelihood ratio test statistics \(T_{n}^{1}\) and \(T_{n}^{2}\) requires maximizing the quasi-log-likelihood functions \(\mathbb{H}_{n}^{1}(\sigma)\) and \(\mathbb{H}_{n}^{2}(\theta)\) over \(\sigma\) and \(\theta\), respectively. This necessitates repeated evaluations of these functions. Since the size of \(S_n\) increases with \(n\), the computation of its inverse can be highly burdensome.

To reduce the computational cost of the matrix inversion, let \(L\) be a positive integer and consider partitioning the observation interval \([0, T_n]\) into \(L\) blocks of equal length. Define
\[
    t_k^L = \frac{kT_n}{L}, \quad k=0,1,\dots,L.
\]
For \(l \in \{1,2\}\) and \(1\leq k\leq L\), let
\[
    \Delta^{l, k} X = \Bigl(\Delta_{i}^{l} X\Bigr)_{i:\, I_i^l \subset (t_{k-1}^L, t_k^L]},
\]
and set $\Delta^{(k)} X = ((\Delta^{1, k}X)^{\top}, (\Delta^{2, k}X)^{\top})^{\top}$.
Denote the covariance matrix corresponding to \(\Delta^{(k)} X\) by \(S_{n}^{(k)}(\sigma)\) and define
\[
    \mathbb{H}_{n}^{1,(L)}(\sigma) = -\frac{1}{2}\sum_{k=1}^{L} \Delta^{(k)} X^{\top} (S_{n}^{(k)})^{-1}(\sigma) \Delta^{(k)} X - \frac{1}{2}\sum_{k=1}^{L} \log\det S_{n}^{(k)}(\sigma),
\]
and
\[
    \hat{\sigma}_n^{(L)} \in \operatorname*{argmax}_{\sigma \in \overline{\Theta}_{1}} \mathbb{H}_{n}^{1,(L)}(\sigma).
\]
Define \(\Delta^{(k)}V(\theta)\) similarly, and set \(\bar{X}^{(k)}(\theta) = \Delta^{(k)} X - \Delta^{(k)}V(\theta)\). Then the quasi-log-likelihood function for \(\theta\) is defined as
\[
    \mathbb{H}_{n}^{2,(L)}(\theta) = -\frac{1}{2}\sum_{k=1}^{L} \bar{X}^{(k)}(\theta)^{\top} (S_{n}^{(k)})^{-1}(\hat{\sigma}_{n}^{(L)}) \bar{X}^{(k)}(\theta).
\]

Since the computational cost for inverting \(S_{n}^{(k)}(\sigma)\) is proportional to the cube of the matrix size, dividing the data into \(L\) blocks roughly reduces this cost by a factor of \(L^{-3}\). However, as the inverse must be computed for each of the \(L\) blocks, the overall computational burden is reduced by approximately a factor of \(L^{-2}\). Moreover, since the quasi-log-likelihood function on each block can be computed independently, there is additional scope for parallelization.

By replacing the observation interval \([0, T_n]\) with \([t_{k-1}^L, t_k^L]\), the results of Ogihara \cite{ogi23} (e.g., Propositions 3.8 and 3.16) continue to hold for \(\mathbb{H}_n^{l,(L)}\). Therefore, the test statistics computed from \(\mathbb{H}_n^{l,(L)}\) satisfy asymptotic properties analogous to those in Theorems~\ref{thm:sigma0}--\ref{thm:theta1}.

\subsubsection{Fast computation of $\log \det S_n$}

When the diffusion coefficient \(b_t\) does not depend on \(t\), the computation of \(\log \det S_n\) can be accelerated as follows. First, define
\[
D=\begin{pmatrix}
\operatorname{diag}\bigl((\tilde{\Sigma}^1_i)_{1\leq i\leq M_1}\bigr) & 0 \\[1mm]
0 & \operatorname{diag}\bigl((\tilde{\Sigma}^2_j)_{1\leq j\leq M_2}\bigr)
\end{pmatrix}.
\]
Then, \(S_n\) can be written as
\[
S_n = D^{1/2}\begin{pmatrix}
\mathcal{E}_{M_1} & \rho G \\[1mm]
\rho G^\top & \mathcal{E}_{M_2}
\end{pmatrix}D^{1/2}.
\]
Furthermore, let \(U\) be an orthogonal matrix and \((\lambda_i)_{i=1}^{M_1+M_2}\) be the eigenvalues satisfying
\[
U^\top \begin{pmatrix} 0 & G \\[1mm] G^\top & 0 \end{pmatrix} U = \operatorname{diag}((\lambda_i)_{i=1}^{M_1+M_2}).
\]
Then it follows that
\[
U^\top \begin{pmatrix} \mathcal{E}_{M_1} & \rho G \\[1mm] \rho G^\top & \mathcal{E}_{M_2} \end{pmatrix} U = \operatorname{diag}((1+\rho \lambda_i)_{i=1}^{M_1+M_2}).
\]
Hence, we obtain
\begin{equation}
\begin{split}
\log \det S_n &= \log \det D + \log \det \begin{pmatrix} \mathcal{E}_{M_1} & \rho G \\[1mm] \rho G^\top & \mathcal{E}_{M_2} \end{pmatrix} \\
&=\sum_{i,l}\log \bigl([\Sigma]_{ll} \lvert I_i^l\rvert\bigr) + \sum_{i=1}^{M_1+M_2}\log \bigl(1+\rho \lambda_i\bigr) \\
&=\sum_{i,l}\log \lvert I_i^l\rvert + \sum_{l=1}^2 M_l\log [\Sigma]_{ll}  + \sum_{i=1}^{M_1+M_2}\log \bigl(1+\rho \lambda_i\bigr).
\end{split}
\end{equation}
Since the first term on the right-hand side does not depend on the parameters, it can be ignored in the optimization process. Moreover, because the eigenvalues \(\lambda_i\) are independent of the parameters, they need to be computed only once and need not be recalculated within the parameter optimization loop. Thus, the above representation can significantly reduce the computational cost of evaluating \(\log \det S_n\).

This acceleration technique for computing \(\log \det S_n\) can also be combined with the block-decomposition of \(S_n\) described in Section~\ref{divide-block-subsubsection}, thereby expediting the computation of \(\log \det S_n^{(k)}\). Furthermore, when \(b_t\) depends on \(t\), the method can be approximately applied by approximating \(b_t\) within each block by \(b_{t_{k-1}^L}\) or a similar representative value.

\section{Numerical Experiments}\label{num-section}

In this section, we conduct numerical experiments using simulated data from diffusion processes to evaluate hypothesis tests based on the quasi-likelihood ratio test statistics \(T_n^1\) and \(T_n^2\) (in practice, the approximated statistics described in Section~\ref{fast-calc-subsection} are employed). We consider three tests: one for \(\sigma\) when the diffusion coefficient is time-invariant, one for \(\sigma\) when the diffusion coefficient is time-varying, and one for \(\theta\). For the time-invariant case for $\sigma$, we also compare the test statistic \(T_n^1\) with a test statistic based on the Hayashi--Yoshida estimator presented in the Appendix.

\subsection{Test for \(\sigma\) with time-invariant diffusion coefficients}\label{subsec:sigma_const}

Consider the two-dimensional diffusion process \((X_{t}^{1}, X_{t}^{2})^{\top}\) governed by the stochastic differential equation
\begin{equation}
    \label{exp:sigma}
    \begin{pmatrix}
    \dd X_{t}^{1} \\
    \dd X_{t}^{2}
    \end{pmatrix}
    =
    \begin{pmatrix}
    \sigma^{(1)} & \sigma^{(3)} \\
    0 & \sigma^{(2)}
    \end{pmatrix}
    \begin{pmatrix}
    \dd W_{t}^{1} \\
    \dd W_{t}^{2}
    \end{pmatrix}, 
    \quad (X_{0}^{1}, X_{0}^{2})^{\top} = (0, 0)^{\top}.
\end{equation}
Let \(\{N_{t}^{1}\}_{t\geq 0}\) and \(\{N_{t}^{2}\}_{t\geq 0}\) be independent Poisson processes with intensities \(\lambda_{1}\) and \(\lambda_{2}\), respectively, and define
\[
    S_i^{n,l} = \inf\{t \geq 0 : N_t^l \ge i\}, \quad (l \in \{1,2\}).
\]
We simulate \(X_{t}^{1}\) and \(X_{t}^{2}\) via an Euler--Maruyama scheme based on a sufficiently fine partition, and then sample the simulated paths at the nonsynchronous time points \(\{S_{i}^{n,1}\}_{i=0}^{M_{1}}\) and \(\{S_{j}^{n,2}\}_{j=0}^{M_{2}}\) generated by the Poisson processes. 

Based on these nonsynchronous observations, we compute the test statistic \(T_{n}^{1}\) and perform the following hypothesis test for the parameter \(\sigma^{(3)}\) in model \eqref{exp:sigma}:
\[
    (\mathrm{T}1)\quad
    \begin{cases}
    H_{0} : \sigma^{(3)} = 0,\\[1mm]
    H_{1} : \text{not } H_{0}.
    \end{cases}
\]
Since \(\sigma^{(3)} = 0\) if and only if \(X_{t}^{1}\) and \(X_{t}^{2}\) are independent, under \(H_{0}\) the processes \(X_{t}^{1}\) and \(X_{t}^{2}\) are independent, and rejection of \(H_{0}\) indicates dependence between \(X_{t}^{1}\) and \(X_{t}^{2}\). We vary the value of \(\sigma^{(3)}\) in the simulations to observe whether the behavior of \(T_{n}^{1}\) aligns with the theoretical predictions.

Furthermore, we consider the test statistic \(V_{n}\) constructed based on the Hayashi--Yoshida estimator (see the Appendix).
Though the Hayashi--Yoshida estimator is the consistent estimator of the quadratic covariation, the quadratic covariation over \([0, T_n]\) is given by
\[
    \langle X^1, X^2 \rangle_{T_n} = \int_{0}^{T_n}[b_{t}b_{t}^{\top}(\sigma)]_{12}\dd t = \sigma^{(2)}\sigma^{(3)}T_n,
\]
under model \eqref{exp:sigma}.
Thus, provided \(\sigma^{(2)} > 0\), we have \(\sigma^{(3)} = 0\) if and only if \(\langle X^1, X^2 \rangle_{T_n} = 0\). 
Hence, 
we can compare the performance of \(T_{n}^{1}\) and \(V_{n}\).

Other parameters are fixed as given in Table~\ref{tab:sigma_const_params}. Here, ``Iteration'' refers to the number of simulation runs (sample size) performed for each value of \(\sigma^{(3)}\) to construct the empirical distribution of the test statistic.

\begin{table}[htb]
\centering
\caption{Parameter settings in Section~\ref{subsec:sigma_const}}
\begin{tabular}{cccccccccc}
    \hline
    Parameter & \(n\) & \(h_{n}\) & \(T_{n}\) & \(\lambda_{1}\) & \(\lambda_{2}\) & \(\sigma^{(1)}\) & \(\sigma^{(2)}\) & Iteration\\
    \hline
    Value & \(10^{3}\) & \(10^{-3}\) & 1 & \(2\times 10^{3}\) & \(3\times 10^{3}\) & 2 & 2 & 500\\
    \hline
\end{tabular}
\label{tab:sigma_const_params}
\end{table}

\begin{rem}
Although the drift is not considered in model \eqref{exp:sigma}, under the high-frequency asymptotic framework used in this study, the effect of drift is asymptotically negligible. Therefore, when testing hypotheses solely concerning \(\sigma\), it suffices to consider the model without the drift.
\end{rem}

\begin{table}[htb]
    \centering
    \caption{Proportion of samples with \(T_{n}^{1} \geq \chi^{2}_{1}(\alpha)\)}
    \begin{tabular}{c|ccc}
      \hline
      \diagbox{True Value}{Threshold} & \(\chi^{2}_{1}(0.10)\) & \(\chi^{2}_{1}(0.05)\) & \(\chi^{2}_{1}(0.01)\) \\
      \hline
      \(\sigma^{(3)}=0\)    & 0.110 & 0.052 & 0.010 \\
      \(\sigma^{(3)}=0.25\) & 0.972 & 0.956 & 0.876 \\
      \(\sigma^{(3)}=0.5\)  & 1.000 & 1.000 & 1.000 \\
      \hline
    \end{tabular}
    \label{tab:const_T}
\end{table}
 
\begin{table}[htb]
    \centering
    \caption{Proportion of samples with \(V_n \in (-z(\alpha\mathbin{/}2),z(\alpha\mathbin{/}2))^c\)}
    \begin{tabular}{c|ccc}
      \hline
      \diagbox{True Value}{Threshold} & \(\pm z(0.10/2)\) & \(\pm z(0.05/2)\) & \(\pm z(0.01/2)\) \\
      \hline
      \(\sigma^{(3)}=0\)    & 0.088 & 0.036 & 0.004 \\
      \(\sigma^{(3)}=0.25\) & 0.846 & 0.748 & 0.524 \\
      \(\sigma^{(3)}=0.5\)  & 1.000 & 1.000 & 1.000 \\
      \hline
    \end{tabular}
    \label{tab:const_V}
\end{table}

Tables~\ref{tab:const_T} and \ref{tab:const_V} report the proportions of samples for which \(T_{n}^{1} \geq \chi^{2}_{1}(\alpha)\) and \(V_{n} \in (-z(\alpha\mathbin{/}2),z(\alpha\mathbin{/}2))^c\), respectively, for various values of \(\sigma^{(3)}\). Here, \
$z(\alpha\mathbin{/}2)$ denotes the \(1-\alpha\mathbin{/}2\) quantile of $\mathcal{N}(0,1)$.
Comparing the results for \(\sigma^{(3)}=0\), we observe that \(T_{n}^{1}\) converges to its asymptotic distribution faster than \(V_{n}\). For \(\sigma^{(3)}=0.5\), both \(T_{n}^{1}\) and \(V_{n}\) reject \(H_{0}\) for all samples, while for \(\sigma^{(3)}=0.25\) the rejection rate is higher for \(T_{n}^{1}\), indicating a difference in power.

\subsubsection*{Experiments under local alternatives}

As discussed in Section~\ref{aump-subsection}, the asymptotic optimality of the test is characterized by the behavior of its power under local alternatives. To compare the asymptotic performance of \(T_{n}^{1}\) and \(V_{n}\), we compute the proportion of samples that reject \(H_0\) when the true value of \(\sigma^{(3)}\) lies in a \(\sqrt{n}\)-neighborhood of 0; that is, we consider
\[
\sigma^{(3)} = 0 + \frac{u}{\sqrt{n}}, \quad u\in\{1,2,3,4\}.
\]

\begin{table}[htb]
    \centering
    \caption{Proportion of samples with \(T_{n}^{1} \geq \chi^2_{1}(\alpha)\) under local alternatives}
    \begin{tabular}{c|ccc}
      \hline
      \diagbox{True Value}{Threshold} & \(\chi^2_{1}(0.10)\) & \(\chi^2_{1}(0.05)\) & \(\chi^2_{1}(0.01)\) \\
      \hline
      \(\sigma^{(3)}=0+1/\sqrt{n}\) & 0.186 & 0.112 & 0.028 \\
      \(\sigma^{(3)}=0+2/\sqrt{n}\) & 0.322 & 0.254 & 0.102 \\
      \(\sigma^{(3)}=0+3/\sqrt{n}\) & 0.424 & 0.292 & 0.118 \\
      \(\sigma^{(3)}=0+4/\sqrt{n}\) & 0.618 & 0.486 & 0.258 \\
      \hline
    \end{tabular}
    \label{tab:local_MLE}
\end{table}

\begin{table}[htb]
    \centering
    \caption{Proportion of samples with \(V_n \in (-z(\alpha\mathbin{/}2),z(\alpha\mathbin{/}2))^c\) under local alternatives}
    \begin{tabular}{c|ccc}
      \hline
      \diagbox{True Value}{Threshold} & \(\pm z(0.10/2)\) & \(\pm z(0.05/2)\) & \(\pm z(0.01/2)\) \\
      \hline
      \(\sigma^{(3)}=0+1/\sqrt{n}\) & 0.108 & 0.054 & 0.010 \\
      \(\sigma^{(3)}=0+2/\sqrt{n}\) & 0.170 & 0.108 & 0.022 \\
      \(\sigma^{(3)}=0+3/\sqrt{n}\) & 0.238 & 0.164 & 0.050 \\
      \(\sigma^{(3)}=0+4/\sqrt{n}\) & 0.376 & 0.276 & 0.106 \\
      \hline
    \end{tabular}
    \label{tab:local_HY}
\end{table}

The results are presented in Tables~\ref{tab:local_MLE} and \ref{tab:local_HY} in the same format as Tables~\ref{tab:const_T} and \ref{tab:const_V}. Note that, as specified in Table~\ref{tab:sigma_const_params}, we set \(n=10^3\) so that \(1/\sqrt{n}\approx0.032\); in other words, the tests are conducted for true values much closer to the null than, for example, \(\sigma^{(3)}=0.25\) in Tables~\ref{tab:const_T} and \ref{tab:const_V}. In every case, \(T_{n}^{1}\) exhibits a lower type II error probability, thereby reflecting the superior asymptotic properties of the test based on \(T_{n}^{1}\).

\subsection{Test for \(\sigma\) with time-varying diffusion coefficients}\label{subsec:sigma_time_dependent}

Consider the two-dimensional diffusion process \((X_{t}^{1}, X_{t}^{2})^{\top}\) governed by
\begin{equation}
    \label{exp:sigma_variable}
    \begin{pmatrix}
    \dd X_{t}^{1} \\
    \dd X_{t}^{2}
    \end{pmatrix}
    =
    \Bigl(1+\frac{\sin(2\pi t)}{2}\Bigr)
    \begin{pmatrix}
    \sigma^{(1)} & \sigma^{(3)} \\
    0 & \sigma^{(2)}
    \end{pmatrix}
    \begin{pmatrix}
    \dd W_{t}^{1} \\
    \dd W_{t}^{2}
    \end{pmatrix}, 
    \quad (X_{0}^{1}, X_{0}^{2})^{\top} = (0,0)^{\top}.
\end{equation}
Equation~\eqref{exp:sigma_variable} is obtained by multiplying the diffusion coefficient in \eqref{exp:sigma} by the periodic, time-varying factor \(1+\sin(2\pi t)/2\). For the parameter \(\sigma^{(3)}\) in model~\eqref{exp:sigma_variable}, we test
\[
    (\mathrm{T}2)\quad
    \begin{cases}
    H_{0}: \sigma^{(3)} = 0,\\[1mm]
    H_{1}: \text{not } H_{0}.
    \end{cases}
\]
The parameter settings are given in Table~\ref{tab:sigma_time_dependent_params}.

\begin{table}[htb]
    \centering
    \caption{Parameter settings in Section~\ref{subsec:sigma_time_dependent}}
    \begin{tabular}{cccccccccc}
        \hline
        Parameter & \(n\) & \(h_{n}\) & \(T_{n}\) & \(\lambda_{1}\) & \(\lambda_{2}\) & \(\sigma^{(1)}\) & \(\sigma^{(2)}\) & Iteration\\
        \hline
        Value & \(10^{3}\) & \(10^{-3}\) & 1 & \(2\times 10^{3}\) & \(3\times 10^{3}\) & 1 & 1 & 500\\
        \hline
    \end{tabular}
    \label{tab:sigma_time_dependent_params}
\end{table}

Table~\ref{tab:local_MLE_time_dependent_sigma} presents the proportions of samples for which \(T_{n}^{1} \ge \chi^2_{1}(\alpha)\) for both the null case and for local alternatives given by
\[
\sigma^{(3)} = 0 + \frac{u}{\sqrt{n}}, \quad u\in\{1,2,3,4\}.
\]
These results confirm that even when the diffusion coefficient varies over time, the behavior of the test statistic is consistent with the theoretical predictions.

\begin{table}[htb]
    \centering
    \caption{Proportion of samples with \(T_{n}^{1} \ge \chi^2_{1}(\alpha)\)}
    \begin{tabular}{c|ccc}
      \hline
      \diagbox{True Value}{Threshold} & \(\chi^2_{1}(0.10)\) & \(\chi^2_{1}(0.05)\) & \(\chi^2_{1}(0.01)\) \\
      \hline
      \(\sigma^{(3)}=0\) & 0.098 & 0.054 & 0.014 \\
      \(\sigma^{(3)}=0+1/\sqrt{n}\) & 0.238 & 0.180 & 0.048 \\
      \(\sigma^{(3)}=0+2/\sqrt{n}\) & 0.612 & 0.472 & 0.218 \\
      \(\sigma^{(3)}=0+3/\sqrt{n}\) & 0.900 & 0.852 & 0.638 \\
      \(\sigma^{(3)}=0+4/\sqrt{n}\) & 0.986 & 0.970 & 0.910 \\
      \hline
    \end{tabular}
    \label{tab:local_MLE_time_dependent_sigma}
\end{table}

\subsection{Test for \(\theta\)}\label{subsec:theta}

Consider the two-dimensional diffusion process \((X_{t}^{1}, X_{t}^{2})^{\top}\) governed by
\begin{equation}
    \label{exp:theta}
    \begin{pmatrix}
    \dd X_{t}^{1} \\
    \dd X_{t}^{2}
    \end{pmatrix}
    =
    \begin{pmatrix}
    \bigl(1+\theta^{(1)}\bigr)\sin t \\
    \bigl(-1+\theta^{(2)}\bigr)\sin t
    \end{pmatrix}
    \dd t
    +
    \begin{pmatrix}
    \sigma^{(1)} & \sigma^{(3)} \\
    0 & \sigma^{(2)}
    \end{pmatrix}
    \begin{pmatrix}
    \dd W_{t}^{1} \\
    \dd W_{t}^{2}
    \end{pmatrix}, 
    \quad (X_{0}^{1}, X_{0}^{2})^{\top} = (0, 0)^{\top}.
\end{equation}
We test the parameter \(\theta^{(1)}\) in model~\eqref{exp:theta} via the hypothesis
\[
    (\mathrm{T}3)\quad
    \begin{cases}
    H_{0}: \theta^{(1)} = 0,\\[1mm]
    H_{1}: \text{not } H_{0}.
    \end{cases}
\]
The parameter settings are given in Table~\ref{tab:theta_params}.

\begin{table}[htb]
    \centering
    \caption{Parameter settings in Section~\ref{subsec:theta}}
    \begin{tabular}{cccccccccccc}
        \hline
        Parameter & \(n\) & \(h_{n}\) & \(T_{n}\) & \(\lambda_{1}\) & \(\lambda_{2}\) & \(\sigma^{(1)}\) & \(\sigma^{(2)}\) & \(\sigma^{(3)}\) & \(\theta^{(2)}\) & Iteration\\
        \hline
        Value & \(2\times10^{6}\) & \(5\times10^{-4}\) & \(10^{3}\) & \(1\times10^{3}\) & \(1.5\times10^{3}\) & 1 & 1 & 0.5 & 0 & 500\\
        \hline
    \end{tabular}
    \label{tab:theta_params}
\end{table}

\begin{table}[htb]
    \centering
    \caption{Proportion of samples with \(T_{n}^{2} \ge \chi^2_{1}(\alpha)\) under local alternatives}
    \begin{tabular}{c|ccc}
      \hline
      \diagbox{True Value}{Threshold} & \(\chi^2_{1}(0.10)\) & \(\chi^2_{1}(0.05)\) & \(\chi^2_{1}(0.01)\) \\
      \hline
      \(\theta^{(1)}=0\) & 0.090 & 0.054 & 0.018 \\
      \(\theta^{(1)}=0+1/\sqrt{nh_{n}}\) & 0.192 & 0.108 & 0.034 \\
      \(\theta^{(1)}=0+2/\sqrt{nh_{n}}\) & 0.406 & 0.310 & 0.122 \\
      \(\theta^{(1)}=0+3/\sqrt{nh_{n}}\) & 0.680 & 0.576 & 0.340 \\
      \(\theta^{(1)}=0+4/\sqrt{nh_{n}}\) & 0.876 & 0.812 & 0.614 \\
      \hline
    \end{tabular}
    \label{tab:local_MLE_theta}
\end{table}

Table~\ref{tab:local_MLE_theta} shows the proportions of samples for which \(T_{n}^{2} \ge \chi^2_{1}(\alpha)\) for the cases \(\theta^{(1)}=0\) and for local alternatives given by
\[
\theta^{(1)} = 0 + \frac{u}{\sqrt{nh_{n}}}, \quad u\in\{1,2,3,4\}.
\]
These results demonstrate that the behavior of the test statistic for the drift parameters is consistent with theoretical predictions.

\section{Proofs of the Main Results}\label{proof-section}

In this section, we present the proofs of the main theorems.

For a matrix \(A\), we denote $A^{\otimes 2} = AA^{\top}$, and for matrices \(A\) and \(B\) of appropriate sizes, we write $A\llbracket B \rrbracket = \tr(AB^{\top})$.

\subsection{Preparations}

Define
\begin{align*}
y_{1,t}(\sigma) &= -\frac{1}{2}\mathcal{A}(\rho_{t})\sum_{l=1}^{2}B_{l,t}^{2} 
+ \mathcal{A}(\rho_{t})\frac{B_{1,t}B_{2,t}\rho_{t,0}}{\rho_{t}} \\
&\quad + \sum_{l=1}^{2}a_{0}^{l}\left(\frac{1}{2} - \frac{1}{2}B_{l,t}^{2} + \log B_{l,t}\right)
+ \int_{\rho_{t,0}}^{\rho_{t}}\frac{\mathcal{A}(\rho)}{\rho}\dd \rho, \\
y_{2,t}(\theta) &= \sum_{p=0}^{\infty}\Biggl\{
-\frac{1}{2}\sum_{l=1}^{2}f_{p}^{ll}\rho_{t,0}^{2p}\phi_{l,t}^{2}(\theta)
+ f_{p}^{12}\rho_{t,0}^{2p+1}\phi_{1,t}(\theta)\phi_{2,t}(\theta)
\Biggr\}.
\end{align*}
Then, define
\[
\mathcal{Y}_{1}(\sigma) = \lim_{T\to\infty}\frac{1}{T}\int_{0}^{T}y_{1,t}(\sigma)\dd t, \quad
\mathcal{Y}_{2}(\theta) = \lim_{T\to\infty}\frac{1}{T}\int_{0}^{T}y_{2,t}(\theta)\dd t.
\]
Propositions 3.8 and 3.16 in Ogihara~\cite{ogi23} show that these quantities correspond to the limits of 
$n^{-1}(\mathbb{H}_{n}^{1}(\sigma) - \mathbb{H}_{n}^{1}(\sigma_{0}))$ and $(nh_n)^{-1}(\mathbb{H}_{n}^{2}(\theta) - \mathbb{H}_{n}^{2}(\theta_{0}))$, respectively, and we obtain
\[
-\partial^{2}_{\sigma}\mathcal{Y}_{1}(\sigma_{0}) = \Gamma_{1}, \quad
-\partial^{2}_{\theta}\mathcal{Y}_{2}(\theta_{0}) = \Gamma_{2}.
\]

When \(1\leq r_1 < d_1\), we partition \(\Gamma_{1}\) by writing
\[
\Gamma_{1} =
\begin{pmatrix*}[c]
\Gamma_{1,1} & \Gamma_{1,2} \\[3mm]
\Gamma_{1,2}^{\top} & \Gamma_{1,3}
\end{pmatrix*},
\]
where \(\Gamma_{1,1}\in\mathbb{R}^{r_{1}\times r_{1}}\), \(\Gamma_{1,2}\in\mathbb{R}^{r_{1}\times (d_{1}-r_{1})}\), and \(\Gamma_{1,3}\in\mathbb{R}^{(d_{1}-r_{1})\times (d_{1}-r_{1})}\). Similarly, we partition \(\Gamma_{2}\) by writing
\[
\Gamma_{2} =
\begin{pmatrix*}[c]
\Gamma_{2,1} & \Gamma_{2,2} \\[3mm]
\Gamma_{2,2}^{\top} & \Gamma_{2,3}
\end{pmatrix*},
\]
with \(\Gamma_{2,1}\in\mathbb{R}^{r_{2}\times r_{2}}\), \(\Gamma_{2,2}\in\mathbb{R}^{r_{2}\times (d_{2}-r_{2})}\), and \(\Gamma_{2,3}\in\mathbb{R}^{(d_{2}-r_{2})\times (d_{2}-r_{2})}\).

Moreover, define
\[
\Lambda_{1} =
\begin{pmatrix}
O & O \\[2mm]
O & (\Gamma_{1,3})^{-1}
\end{pmatrix}, \quad
\Lambda_{2} =
\begin{pmatrix}
O & O \\[2mm]
O & (\Gamma_{2,3})^{-1}
\end{pmatrix}.
\]

\subsection{Proof of Theorems~\ref{thm:sigma0} and \ref{thm:sigma1}}

To prove Theorems~\ref{thm:sigma0} and \ref{thm:sigma1}, we establish four lemmas. The first two lemmas concern the convergence of \(\partial_\sigma^2 \mathbb{H}_n^l\).

For \(x,y\in\Theta_{1}\), define
\[
\tilde{\Gamma}_{1, n}^l(x, y) = -(l+1)\int_{0}^{1} \frac{(1-u)^l}{n}\,\partial^{2}_{\sigma} \mathbb{H}_{n}^{1}\Bigl(y + u(x - y)\Bigr)\dd u
\]
for $l\in \{0,1\}$.
Then, similarly to \(\Gamma_{1}\), we decompose \(\tilde{\Gamma}_{1, n}^l(x, y)\) as
\[
\tilde{\Gamma}_{1, n}^l(x, y) = 
\begin{pmatrix*}[c]
\tilde{\Gamma}_{1,1,n}^l & \tilde{\Gamma}_{1,2,n}^l \\[3mm]
(\tilde{\Gamma}_{1,2,n}^l)^{\top} & \tilde{\Gamma}_{1,3,n}^l
\end{pmatrix*}(x, y).
\]

In the process of applying Taylor's theorem to both $\mathbb{H}_n^1$ and its derivative $\partial_\sigma \mathbb{H}_n^1$, the term $\tilde{\Gamma}_{1,n}^l$ arises. Therefore, to analyze the asymptotic behavior of $T_n^1$, we first establish the convergence in probability of $\tilde{\Gamma}_{1,n}^l$.

\begin{lem}\label{lem:1}
Assume conditions (A1)--(A4) and (A6). Then, under \(H_{0}^{(\sigma)}\), 
$\tilde{\Gamma}_{1,n}^1(\tilde{\sigma}_n,\hat{\sigma}_n)
\overset{P}{\rightarrow} \Gamma_{1}$ as $n\to\infty$.

\begin{proof}
Define
$\sigma_{u,n} = \hat{\sigma}_{n} + u\bigl(\tilde{\sigma}_n - \hat{\sigma}_{n}\bigr)$.
Proposition~3.8 in Ogihara~\cite{ogi23} yields
\begin{equation}\label{lem:1-eq1}
\begin{split}
&\Biggl|\,2\int_{0}^{1} \frac{1-u}{n}\,\partial^{2}_{\sigma} \mathbb{H}_{n}^{1}(\sigma_{u,n})\dd u
- 2\int_{0}^{1} (1-u)\,\partial^{2}_{\sigma} \mathcal{Y}_{1}(\sigma_{u,n})\dd u \,\Biggr|\\[1mm]
&\quad \leq 2\int_{0}^{1} (1-u)\dd u\cdot \sup_{\sigma\in\Theta_{1}}\Biggl|\frac{1}{n}\,\partial^{2}_{\sigma} \mathbb{H}_{n}^{1}(\sigma)
- \partial^{2}_{\sigma} \mathcal{Y}_{1}(\sigma)\Biggr|\\[1mm]
&\quad \overset{P}{\rightarrow} 0.
\end{split}
\end{equation}
Furthermore,
\begin{equation}\label{lem:1-eq2}
\begin{split}
&\Biggl|\, -2\int_{0}^{1} (1-u)\,\partial^{2}_{\sigma} \mathcal{Y}_{1}(\sigma_{u,n})\dd u - \Gamma_{1}\,\Biggr| \\
&\quad = \Biggl|\, 2\int_{0}^{1} (1-u)\,\partial^{2}_{\sigma} \mathcal{Y}_{1}(\sigma_{u,n})\dd u
- 2\int_{0}^{1} (1-u)\,\partial^{2}_{\sigma} \mathcal{Y}_{1}(\sigma_{0})\dd u \,\Biggr|\\[1mm]
&\quad = \Biggl|\, 2\int_{0}^{1} (1-u)
\left(\int_{0}^{1}\partial^{3}_{\sigma}\mathcal{Y}_{1}\Bigl(\sigma_{0} + v(\sigma_{u,n} - \sigma_{0})\Bigr)\dd v\right)
\cdot (\sigma_{u,n}-\sigma_{0})\dd u \,\Biggr|.
\end{split}
\end{equation}
By (A1), the \(k\)th derivative of \(\mathcal{Y}_{1}(\sigma)\) (for \(k\in\{0,1,2,3,4\}\)) is bounded; that is, there exists a constant \(K'>0\) such that
\begin{equation}\label{lem:1-eq3}
\left|\partial^{k}_{\sigma}\mathcal{Y}_{1}\Bigl(\sigma_{0} + v(\sigma_{u,n} - \sigma_{0})\Bigr)\right|\leq K', \quad (k\in \{0,1,2,3,4\}).
\end{equation}
Moreover, by the consistency of \(\hat{\sigma}_{n}\) and of \(\tilde{\sigma}_{n}\) under \(H_{0}^{(\sigma)}\) (see Ogihara~\cite{ogi23}, Section 3.2), we have
\begin{align*}
\Biggl|\, 2\int_{0}^{1} (1-u)&\left(\int_{0}^{1}\partial^{3}_{\sigma}\mathcal{Y}_{1}\Bigl(\sigma_{0} + v(\sigma_{u,n} - \sigma_{0})\Bigr)\dd v\right)(\sigma_{u,n}-\sigma_{0})\dd u \,\Biggr|\\[1mm]
&\leq 2K'\int_{0}^{1} (1-u)\Bigl(u\bigl|\tilde{\sigma}_n - \sigma_{0}\bigr| + (1-u)\bigl|\hat{\sigma}_n - \sigma_{0}\bigr|\Bigr)\dd u\\[1mm]
&\leq K'\Bigl(\bigl|\tilde{\sigma}_n - \sigma_{0}\bigr| + \bigl|\hat{\sigma}_n - \sigma_{0}\bigr|\Bigr)
\overset{P}{\rightarrow} 0.
\end{align*}
Together with \eqref{lem:1-eq1} and \eqref{lem:1-eq2}, it follows that
\[
\Biggl|\,-2\int_{0}^{1} \frac{1-u}{n}\,\partial^{2}_{\sigma} \mathbb{H}_{n}^{1}(\sigma_{u,n})\dd u - \Gamma_{1}\,\Biggr|
\overset{P}{\rightarrow} 0.
\]
\end{proof}
\end{lem}

\begin{lem}\label{lem:2}
Assume conditions (A1) -- (A4) and (A6). Under \(H_{0}^{(\sigma)}\), we have
$\tilde{\Gamma}_{1, n}^0(\tilde{\sigma}_{n}, \sigma_{0})
\overset{P}{\rightarrow}\Gamma_{1}$, 
and
$\tilde{\Gamma}_{1, n}^0(\tilde{\sigma}_{n}, \hat{\sigma}_{n}) 
\overset{P}{\rightarrow}\Gamma_{1}$ as $n\to\infty$.
\end{lem}

\begin{proof}
Define
$\sigma_{u,0,n} = \sigma_{0} + u\bigl(\tilde{\sigma}_{n} - \sigma_{0}\bigr)$.
By an argument similar to that in the proof of Lemma~\ref{lem:1}, we have
\[
\left|
\int_{0}^{1} \frac{1}{n}\,\partial^{2}_{\sigma} \mathbb{H}_{n}^{1}(\sigma_{u,0,n})\,\mathrm{d}u
-
\int_{0}^{1} \partial^{2}_{\sigma} \mathcal{Y}_{1}(\sigma_{u,0,n})\,\mathrm{d}u
\right|
\overset{P}{\rightarrow} 0,
\]
and
\begin{align*}
&\Biggl|
-\int_{0}^{1} \partial^{2}_{\sigma} \mathcal{Y}_{1}(\sigma_{u,0,n})\,\mathrm{d}u
-\Gamma_{1}
\Biggr| \\
&\quad =
\left|
\int_{0}^{1} \partial^{2}_{\sigma} \mathcal{Y}_{1}(\sigma_{u,0,n})\,\mathrm{d}u
-
\int_{0}^{1} \partial^{2}_{\sigma} \mathcal{Y}_{1}(\sigma_{0})\,\mathrm{d}u
\right|\\[1mm]
&\quad = \left|
\int_{0}^{1} \left(\int_{0}^{1}\partial^{3}_{\sigma}\mathcal{Y}_{1}\Bigl(\sigma_{0} + v\bigl(\sigma_{u,0,n} - \sigma_{0}\bigr)\Bigr)\,\mathrm{d}v\right)
\bigl(\sigma_{u,0,n}-\sigma_{0}\bigr)\,\mathrm{d}u
\right|\\[1mm]
&\quad \leq K'\,\bigl|\hat{\sigma}_{n} - \sigma_{0}\bigr|
\overset{P}{\rightarrow} 0. 
\end{align*}
Thus, we obtain $\tilde{\Gamma}_{1, n}^0(\tilde{\sigma}_{n}, \sigma_{0})
\overset{P}{\rightarrow}\Gamma_{1}$.
A similar argument applies to \(\tilde{\Gamma}_{1,n}^0(\tilde{\sigma}_{n}, \hat{\sigma}_{n})\).
\end{proof}

Define the projection \(P_{0}:\mathbb{R}^{d_{1}} \to \mathbb{R}^{d_{1}-r_{1}}\) onto the components \(r_{1}+1, \ldots, d_{1}\) by
\[
P_{0}\begin{pmatrix}
\sigma^{(1)} \\ \sigma^{(2)} \\ \vdots \\ \sigma^{(d_{1})}
\end{pmatrix}
=
\begin{pmatrix}
\sigma^{(r_{1}+1)} \\ \vdots \\ \sigma^{(d_{1})}
\end{pmatrix}.
\]

\begin{lem}\label{lem:4}
Assume \(1\leq r_1<d_1\) and conditions (A1) -- (A4) and (A6). Under \(H_{0}^{(\sigma)}\) and provided that \(\tilde{\sigma}_n\in \Theta_{1}\), we have
\[
\frac{1}{\sqrt{n}}\partial_{P_{0}\sigma}\mathbb{H}_{n}^{1}(\sigma_{0})
=\tilde{\Gamma}_{1,3,n}^0(\tilde{\sigma}_{n}, \sigma_{0})\,\sqrt{n}\,P_{0}\bigl(\tilde{\sigma}_{n} - \sigma_{0}\bigr).
\]
\begin{proof}
Under \(H_{0}^{(\sigma)}\), note that the first \(r_1\) components of \(\tilde{\sigma}_{n}\) and \(\sigma_{0}\) coincide. Expanding with respect to \(P_{0}\sigma\) gives
\begin{equation}\label{eq:1}
\partial_{P_{0}\sigma}\mathbb{H}_{n}^{1}(\tilde{\sigma}_{n}) - \partial_{P_{0}\sigma}\mathbb{H}_{n}^{1}(\sigma_{0})
=\left(\int_{0}^{1}\partial^{2}_{P_{0}\sigma}\mathbb{H}_{n}^{1}\Bigl(\sigma_{0} + u\bigl(\tilde{\sigma}_{n} - \sigma_{0}\bigr)\Bigr)\,\mathrm{d}u\right)
P_{0}\bigl(\tilde{\sigma}_{n} - \sigma_{0}\bigr).
\end{equation}
By the definition of \(\tilde{\Gamma}_{1,3,n}^0(x,y)\),
\[
-\int_{0}^{1} \frac{1}{n}\,\partial^{2}_{P_{0}\sigma}\mathbb{H}_{n}^{1}\Bigl(\sigma_{0} + u\bigl(\tilde{\sigma}_{n} - \sigma_{0}\bigr)\Bigr)\,\mathrm{d}u
=\tilde{\Gamma}_{1,3,n}^0(\tilde{\sigma}_{n}, \sigma_{0}).
\]
Hence, combining with \eqref{eq:1}, we obtain
\[
\frac{1}{\sqrt{n}}\partial_{P_{0}\sigma}\mathbb{H}_{n}^{1}(\tilde{\sigma}_{n})
-\frac{1}{\sqrt{n}}\partial_{P_{0}\sigma}\mathbb{H}_{n}^{1}(\sigma_{0})
=-\tilde{\Gamma}_{1,3,n}^0(\tilde{\sigma}_{n}, \sigma_{0})\,\sqrt{n}\,P_{0}\bigl(\tilde{\sigma}_{n} - \sigma_{0}\bigr).
\]
Since \(\tilde{\sigma}_{n}\in \Theta_{1}\) implies \(\partial_{P_{0}\sigma}\mathbb{H}_{n}^{1}(\tilde{\sigma}_{n})=0\), the claim follows.
\end{proof}
\end{lem}

\begin{lem}\label{lem:5}
Assume \(1\leq r_1<d_1\) and let \(Z\sim \mathcal{N}(0, \mathcal{E}_{d_{1}})\). Under conditions (A1) -- (A4) and (A6), under \(H_{0}^{(\sigma)}\) we have
\[
\sqrt{n}\bigl(\tilde{\sigma}_{n} - \hat{\sigma}_{n}\bigr)
\overset{d}{\rightarrow} -\Gamma_{1}^{-1}\bigl(\mathcal{E}_{d_{1}} - \Gamma_{1}\Lambda_{1}\bigr)\Gamma_{1}^{1/2}Z
\]
as $n\to\infty$.
\begin{proof}
Considering a Taylor expansion in \(\sigma\), we have
\begin{equation}\label{lem:5-eq1}
\begin{split}
\frac{1}{\sqrt{n}}\partial_{\sigma}\mathbb{H}_{n}^{1}(\tilde{\sigma}_{n})
-\frac{1}{\sqrt{n}}\partial_{\sigma}\mathbb{H}_{n}^{1}(\sigma_{0})
=-\tilde{\Gamma}_{1,n}^0(\tilde{\sigma}_{n}, \sigma_{0})\,\sqrt{n}\bigl(\tilde{\sigma}_{n}-\sigma_{0}\bigr).
\end{split}
\end{equation}
Define
\[
\tilde{\Lambda}_{1,n}(\tilde{\sigma}_{n}, \sigma_{0}) = 
\begin{pmatrix}
O & O \\[2mm]
O & (\tilde{\Gamma}_{1,3,n}^0)^{-1}(\tilde{\sigma}_{n}, \sigma_{0})
\end{pmatrix}\mathbf{1}_{\left\{\det \tilde{\Gamma}_{1,3,n}^0(\tilde{\sigma}_{n}, \sigma_{0}) > 0\right\}}.
\]
Then, by Lemma~\ref{lem:4} (assuming \(\tilde{\sigma}_{n}\in \Theta_{1}\)),
\[
\sqrt{n}\bigl(\tilde{\sigma}_{n} - \sigma_{0}\bigr)
=\tilde{\Lambda}_{1,n}(\tilde{\sigma}_{n}, \sigma_{0})\frac{1}{\sqrt{n}}\partial_{\sigma}\mathbb{H}_{n}^{1}(\sigma_{0})
\quad \text{on } \left\{\det \tilde{\Gamma}_{1,3,n}^0 (\tilde{\sigma}_{n}, \sigma_{0}) > 0\right\}.
\]
Thus, combining with \eqref{lem:5-eq1},
\[
\frac{1}{\sqrt{n}}\partial_{\sigma}\mathbb{H}_{n}^{1}(\tilde{\sigma}_{n})
-\frac{1}{\sqrt{n}}\partial_{\sigma}\mathbb{H}_{n}^{1}(\sigma_{0})
=-\tilde{\Gamma}_{1,n}^0 (\tilde{\sigma}_{n}, \sigma_{0})
\tilde{\Lambda}_{1,n}(\tilde{\sigma}_{n}, \sigma_{0})\frac{1}{\sqrt{n}}\partial_{\sigma}\mathbb{H}_{n}^{1}(\sigma_{0})
\]
on \(\{\det \tilde{\Gamma}_{1,3,n}^0 (\tilde{\sigma}_{n}, \sigma_{0}) > 0\}\). Hence,
\[
\frac{1}{\sqrt{n}}\partial_{\sigma}\mathbb{H}_{n}^{1}(\tilde{\sigma}_{n})
=(\mathcal{E}_{d_{1}} - \tilde{\Gamma}_{1,n}^0 \tilde{\Lambda}_{1,n}(\tilde{\sigma}_{n}, \sigma_{0}))
\frac{1}{\sqrt{n}}\partial_{\sigma}\mathbb{H}_{n}^{1}(\sigma_{0})
\quad \text{on } \{\det \tilde{\Gamma}_{1,3,n}^0 (\tilde{\sigma}_{n}, \sigma_{0}) > 0\}.
\]
By Lemma~\ref{lem:2} and the positive definiteness of  \(\Gamma_{1}\), we obtain
$P\Bigl(\det \tilde{\Gamma}_{1,3,n}^0 (\tilde{\sigma}_{n}, \sigma_{0}) \leq 0\Bigr) \to 0$,
which implies
$\tilde{\Lambda}_{1,n}(\tilde{\sigma}_{n}, \sigma_{0}) \overset{P}{\rightarrow} \Lambda_{1}$.
Moreover, by Ogihara~\cite{ogi23} Proposition~3.14,
$n^{-1/2}\partial_{\sigma}\mathbb{H}_{n}^{1}(\sigma_{0})
\overset{d}{\rightarrow} \mathcal{N}(0,\Gamma_{1})$.
Thus, letting \(Z\sim \mathcal{N}(0,\mathcal{E}_{d_{1}})\),
\begin{equation}\label{eq:2}
\frac{1}{\sqrt{n}}\partial_{\sigma}\mathbb{H}_{n}^{1}(\tilde{\sigma}_{n})
\overset{d}{\rightarrow} (\mathcal{E}_{d_{1}} - \Gamma_{1}\Lambda_{1})\Gamma_{1}^{1/2}Z.
\end{equation}
On the other hand, by the consistency of \(\hat{\sigma}_{n}\) (which implies \(P(\hat{\sigma}_n\in\Theta_{1})\to 1\)) and noting that \(\partial_{\sigma}\mathbb{H}_{n}^{1}(\hat{\sigma}_{n}) = 0\) whenever \(\hat{\sigma}_{n}\in\Theta_{1}\), a Taylor expansion yields
\begin{align*}
\frac{1}{\sqrt{n}}\partial_{\sigma}\mathbb{H}_{n}^{1}(\tilde{\sigma}_{n})
-\frac{1}{\sqrt{n}}\partial_{\sigma}\mathbb{H}_{n}^{1}(\hat{\sigma}_{n})
=-\tilde{\Gamma}_{1,n}^0(\tilde{\sigma}_{n}, \hat{\sigma}_{n})\,\sqrt{n}\bigl(\tilde{\sigma}_{n} - \hat{\sigma}_{n}\bigr).
\end{align*}
Thus,
\begin{equation}\label{eq:3}
\sqrt{n}\bigl(\tilde{\sigma}_{n} - \hat{\sigma}_{n}\bigr)
=-(\tilde{\Gamma}_{1,n}^0)^{-1}(\tilde{\sigma}_{n}, \hat{\sigma}_{n})\frac{1}{\sqrt{n}}\partial_{\sigma}\mathbb{H}_{n}^{1}(\tilde{\sigma}_{n})
\quad \text{on } \{\det \tilde{\Gamma}_{1,n}^0 (\tilde{\sigma}_{n}, \hat{\sigma}_{n})>0\}.
\end{equation}
Again by Lemma~\ref{lem:2},
$P\Bigl(\det \tilde{\Gamma}_{1,n}^0 (\tilde{\sigma}_{n}, \hat{\sigma}_{n}) \leq 0\Bigr) \to 0$,
which implies that
\[
(\tilde{\Gamma}_{1,n}^0)^{-1}(\tilde{\sigma}_{n}, \hat{\sigma}_{n})\mathbf{1}_{\{\det \tilde{\Gamma}_{1,n}^0 (\tilde{\sigma}_{n}, \hat{\sigma}_{n})>0\}}
\overset{P}{\rightarrow} \Gamma_{1}^{-1}.
\]
Combining \eqref{eq:2} and \eqref{eq:3}, we obtain
\[
\sqrt{n}\bigl(\tilde{\sigma}_{n} - \hat{\sigma}_{n}\bigr)
\overset{d}{\rightarrow} -\Gamma_{1}^{-1}\bigl(\mathcal{E}_{d_{1}} - \Gamma_{1}\Lambda_{1}\bigr)\Gamma_{1}^{1/2}Z, \quad (n\to\infty).
\]
\end{proof}
\end{lem}

\subsubsection*{Proof of Theorem \ref{thm:sigma0}}
By Taylor's theorem, we have
\begin{align*}
\mathbb{H}_{n}^{1}(\tilde{\sigma}_{n}) 
= \mathbb{H}_{n}^{1}(\hat{\sigma}_{n}) + \partial_\sigma \mathbb{H}_{n}^{1}(\hat{\sigma}_{n})\bigl(\tilde{\sigma}_n - \hat{\sigma}_n\bigr) 
-\frac{1}{2} \tilde{\Gamma}_{1,n}^1(\tilde{\sigma}_n,\hat{\sigma}_n)\Bigl\llbracket \Bigl[\sqrt{n}(\tilde{\sigma}_n - \hat{\sigma}_n)\Bigr]^{\otimes 2} \Bigr\rrbracket.
\end{align*}
Since \(\hat{\sigma}_n\) is consistent (i.e., \(P(\hat{\sigma}_n\in\Theta_{1})\to1\)) and, when \(\hat{\sigma}_n\in\Theta_{1}\), we have \(\partial_{\sigma}\mathbb{H}_{n}^{1}(\hat{\sigma}_n)=0\), it follows that
\begin{equation}\label{eq:4}
\mathbb{H}_{n}^{1}(\tilde{\sigma}_{n}) - \mathbb{H}_{n}^{1}(\hat{\sigma}_{n})
=-\frac{1}{2} \tilde{\Gamma}_{1,n}^1(\tilde{\sigma}_n,\hat{\sigma}_n)\Bigl\llbracket \Bigl[\sqrt{n}(\tilde{\sigma}_n - \hat{\sigma}_n)\Bigr]^{\otimes 2} \Bigr\rrbracket.
\end{equation}

We first consider the case \(r_{1} = d_{1}\). In this situation, \(\tilde{\sigma}_n = \sigma_{0} = 0\). Let \(Z\sim\mathcal{N}(0,\mathcal{E}_{d_{1}})\). Then, by Ogihara~\cite{ogi23} Theorem~2.3,
$\sqrt{n}(\hat{\sigma}_n - \sigma_{0}) \overset{d}{\rightarrow} \Gamma_{1}^{-1/2}Z$.
Thus,
\begin{equation}\label{eq:5}
\sqrt{n}(\tilde{\sigma}_n - \hat{\sigma}_n) \overset{d}{\rightarrow} \Gamma_{1}^{-1/2}Z.
\end{equation}
Then, using \eqref{eq:4}, \eqref{eq:5}, Lemma~\ref{lem:1}, and Slutsky\textquotesingle s theorem, we obtain
\[
T_{n}^{1} = 2\Bigl(\mathbb{H}_{n}^{1}(\hat{\sigma}_n)-\mathbb{H}_{n}^{1}(\tilde{\sigma}_n)\Bigr)
\overset{d}{\rightarrow} Z^{\top}\Gamma_{1}^{-1/2}\Gamma_{1}\Gamma_{1}^{-1/2}Z
\sim \chi^{2}_{d_{1}} = \chi^{2}_{r_{1}}.
\]

Next, consider the case \(1\le r_{1} < d_{1}\). Using \eqref{eq:4}, along with Lemmas~\ref{lem:1} and \ref{lem:5} and applying Slutsky\textquotesingle s theorem, we have
\begin{align*}
T_{n}^{1} &= -2\Bigl(\mathbb{H}_{n}^{1}(\tilde{\sigma}_n)-\mathbb{H}_{n}^{1}(\hat{\sigma}_n)\Bigr)\\[1mm]
&= \tilde{\Gamma}_{1,n}^1(\tilde{\sigma}_n, \hat{\sigma}_n)
\Bigl\llbracket\Bigl[\sqrt{n}(\tilde{\sigma}_n-\hat{\sigma}_n)\Bigr]^{\otimes 2} \Bigr\rrbracket\\[1mm]
&\overset{d}{\rightarrow} Z^{\top}\Gamma_{1}^{1/2}\bigl(\mathcal{E}_{d_{1}}-\Gamma_{1}\Lambda_{1}\bigr)^{\top}\Gamma_{1}^{-1}\bigl(\mathcal{E}_{d_{1}}-\Gamma_{1}\Lambda_{1}\bigr)\Gamma_{1}^{1/2}Z\\[1mm]
&= Z^{\top}\Gamma_{1}^{1/2}\bigl(\Gamma_{1}^{-1}-\Lambda_{1}\bigr)\Gamma_{1}^{1/2}Z.
\end{align*}
Here we used the fact that
\[
\Lambda_{1}\Gamma_{1}\Lambda_{1} =
\begin{pmatrix}
O & O \\[2mm]
O & (\Gamma_{1,3})^{-1}
\end{pmatrix}
\begin{pmatrix*}[c]
\Gamma_{1,1} & \Gamma_{1,2} \\[2mm]
\Gamma_{1,2}^{\top} & \Gamma_{1,3}
\end{pmatrix*}
\begin{pmatrix}
O & O \\[2mm]
O & (\Gamma_{1,3})^{-1}
\end{pmatrix}
=\Lambda_{1}.
\]
Moreover, noting that
\begin{align*}
\Gamma_{1}^{1/2}\bigl(\Gamma_{1}^{-1}-\Lambda_{1}\bigr)\Gamma_{1}^{1/2}\Gamma_{1}^{1/2}\bigl(\Gamma_{1}^{-1}-\Lambda_{1}\bigr)\Gamma_{1}^{1/2}
&=\Gamma_{1}^{1/2}\bigl(\Gamma_{1}^{-1}-\Lambda_{1}\bigr)(\mathcal{E}_{d_{1}}-\Gamma_{1}\Lambda_{1})\Gamma_{1}^{1/2}\\[1mm]
&=\Gamma_{1}^{1/2}\bigl(\Gamma_{1}^{-1}-\Lambda_{1}\bigr)\Gamma_{1}^{1/2},
\end{align*}
we see that \(\Gamma_{1}^{1/2}\bigl(\Gamma_{1}^{-1}-\Lambda_{1}\bigr)\Gamma_{1}^{1/2}\) is a projection matrix. Its rank, equal to its trace, is given by
\begin{align*}
\mathrm{tr}\Bigl\{\Gamma_{1}^{1/2}\bigl(\Gamma_{1}^{-1}-\Lambda_{1}\bigr)\Gamma_{1}^{1/2}\Bigr\}
=\mathrm{tr}\Bigl(\mathcal{E}_{d_{1}}-\Gamma_{1}\Lambda_{1}\Bigr)
=\mathrm{tr}\begin{pmatrix*}[c]
\mathcal{E}_{r_{1}} & * \\[2mm]
O & O
\end{pmatrix*}
= r_{1},
\end{align*}
which leads to the conclusion.

\qed

To prove Theorem \ref{thm:sigma1}, we introduce an additional lemma. The function \(\mathcal{Y}_1(\sigma)\) was the probability limit of
$n^{-1}\left(\mathbb{H}_n^1(\sigma) - \mathbb{H}_n^1(\sigma_0)\right)$,
and this convergence remains valid even when \(\sigma_0\) is replaced by \(\hat{\sigma}_n\).

\begin{lem}\label{lem:3}
Assume conditions (A1) -- (A4) and (A6). Then,
\[
\sup_{\sigma\in\Theta_{1}}\left|\frac{1}{n}\Bigl(\mathbb{H}_{n}^{1}(\sigma)-\mathbb{H}_{n}^{1}(\hat{\sigma}_{n})\Bigr)-\mathcal{Y}_{1}(\sigma)\right|
\overset{P}{\rightarrow}0, \quad (n\to\infty).
\]
\begin{proof}
We have
\begin{align*}
\left|\frac{1}{n}\Bigl(\mathbb{H}_{n}^{1}(\sigma)-\mathbb{H}_{n}^{1}(\hat{\sigma}_{n})\Bigr)-\mathcal{Y}_{1}(\sigma)\right|
&\le \left|\frac{1}{n}\Bigl(\mathbb{H}_{n}^{1}(\sigma)-\mathbb{H}_{n}^{1}(\sigma_{0})\Bigr)-\mathcal{Y}_{1}(\sigma)\right|\\[1mm]
&\quad +\left|\frac{1}{n}\Bigl(\mathbb{H}_{n}^{1}(\hat{\sigma}_{n})-\mathbb{H}_{n}^{1}(\sigma_{0})\Bigr)
-\mathcal{Y}_{1}(\hat{\sigma}_{n})\right|
+\left|\mathcal{Y}_{1}(\hat{\sigma}_{n})\right|.
\end{align*}
By Ogihara~\cite{ogi23} Proposition~3.8,
\[
\sup_{\sigma\in\Theta_{1}}\left|\frac{1}{n}\Bigl(\mathbb{H}_{n}^{1}(\sigma)-\mathbb{H}_{n}^{1}(\sigma_{0})\Bigr)-\mathcal{Y}_{1}(\sigma)\right|
\overset{P}{\rightarrow}0.
\]
Thus, it suffices to show that
$-\mathcal{Y}_{1}(\hat{\sigma}_{n})\overset{P}{\rightarrow}0$.
Since \(\mathcal{Y}_{1}(\sigma_{0})=0\) and by the boundedness of the derivatives of \(\mathcal{Y}_{1}(\sigma)\) (cf. \eqref{lem:1-eq3}) along with the consistency of \(\hat{\sigma}_n\), we have
\begin{align*}
0\le -\mathcal{Y}_{1}(\hat{\sigma}_{n})
&=\left|\mathcal{Y}_{1}(\hat{\sigma}_{n})-\mathcal{Y}_{1}(\sigma_{0})\right|\\[1mm]
&=\left|\left(\int_{0}^{1}\partial_{\sigma}\mathcal{Y}_{1}\Bigl(\sigma_{0}+u\bigl(\hat{\sigma}_{n}-\sigma_{0}\bigr)\Bigr)\dd u\right)
\bigl(\hat{\sigma}_{n}-\sigma_{0}\bigr)\right|\\[1mm]
&\le K'\left|\hat{\sigma}_{n}-\sigma_{0}\right|\\[1mm]
&\overset{P}{\rightarrow}0.
\end{align*}
This completes the proof.
\end{proof}
\end{lem}

\subsubsection*{Proof of Theorem \ref{thm:sigma1}}
By Proposition~3.9 in Ogihara~\cite{ogi23}, Assumption (A6), and Remark~4 in Ogihara and Yoshida~\cite{ogi-yos14}, there exists a positive constant \(c\) such that
$\mathcal{Y}_{1}(\sigma) \le -c\,\lvert\sigma - \sigma_{0}\rvert^{2}$.
Under \(H_{1}^{(\sigma)}\), there exists some \(\ell_{1}\in\{1,\ldots,r_{1}\}\) for which \(\sigma_{0}^{(\ell_{1})}\neq 0\). Therefore,
\[
\frac{1}{2n}T_{n}^{1} + \left|\frac{1}{2n}T_{n}^{1} + \mathcal{Y}_{1}(\tilde{\sigma}_{n})\right|
\ge -\mathcal{Y}_{1}(\tilde{\sigma}_{n})
\ge c\,\lvert\tilde{\sigma}_{n} - \sigma_{0}\rvert^{2}
\ge c\,\lvert\sigma_{0}^{(\ell_{1})}\rvert^{2} > 0.
\]
Hence,
\begin{align*}
P\bigl(T_{n}^{1} \le M\bigr)
&= P\Biggl(\frac{1}{2n}T_{n}^{1} + \left|\frac{1}{2n}T_{n}^{1} + \mathcal{Y}_{1}(\tilde{\sigma}_{n})\right|
\le \frac{M}{2n} + \left|\frac{1}{2n}T_{n}^{1} + \mathcal{Y}_{1}(\tilde{\sigma}_{n})\right|\Biggr)\\[1mm]
&\le P\Biggl(c\,\lvert\sigma_{0}^{(\ell_{1})}\rvert^{2}
\le \frac{M}{2n} + \Biggl|\frac{1}{n}\Bigl(\mathbb{H}_{n}^{1}(\hat{\sigma}_{n}) - \mathbb{H}_{n}^{1}(\tilde{\sigma}_{n})\Bigr) + \mathcal{Y}_{1}(\tilde{\sigma}_{n})\Biggr|\Biggr)\\[1mm]
&\le P\Biggl(c\,\lvert\sigma_{0}^{(\ell_{1})}\rvert^{2} - \frac{M}{2n}
\le \sup_{\sigma\in\Theta_{1}}\Biggl|\frac{1}{n}\Bigl(\mathbb{H}_{n}^{1}(\sigma) - \mathbb{H}_{n}^{1}(\hat{\sigma}_{n})\Bigr) + \mathcal{Y}_{1}(\sigma)\Biggr|\Biggr)\\[1mm]
&\le P\Biggl(\frac{1}{2}c\,\lvert\sigma_{0}^{(\ell_{1})}\rvert^{2}
\le \sup_{\sigma\in\Theta_{1}}\Biggl|\frac{1}{n}\Bigl(\mathbb{H}_{n}^{1}(\sigma) - \mathbb{H}_{n}^{1}(\hat{\sigma}_{n})\Bigr) - \mathcal{Y}_{1}(\sigma)\Biggr|\Biggr) + o(1).
\end{align*}
By Lemma~\ref{lem:3}, the right-hand side converges to 0, which completes the proof.

\qed

\subsection{Proofs of Theorems \ref{thm:theta0} and \ref{thm:theta1}}

Next, we prove Theorems \ref{thm:theta0} and \ref{thm:theta1}. The basic structure of the proof is analogous to that of Theorems \ref{thm:sigma0} and \ref{thm:sigma1}. In Proposition 3.16 in Ogihara~\cite{ogi23}, a result on the uniform convergence in \(\theta\) of
$(nh_n)^{-1}\left(\mathbb{H}_n^2(\theta)-\mathbb{H}_n^2(\theta_0)\right)$
is established, and we employ this result in our proof.

For \(x,y\in\Theta_{2}\), define
\[
\tilde{\Gamma}_{2,n}^l(x,y) = -(l+1)\int_{0}^{1} \frac{(1-u)^l}{nh_{n}}\,\partial^{2}_{\theta} \mathbb{H}_{n}^{2}\Bigl(y+u\bigl(x-y\bigr)\Bigr)\,\mathrm{d}u
\]
for $l\in \{0,1\}$,
and decompose it as
\[
\tilde{\Gamma}_{2,n}^l(x,y) =
\begin{pmatrix*}[c]
\tilde{\Gamma}_{2,1,n}^l & \tilde{\Gamma}_{2,2,n}^l \\[2mm]
(\tilde{\Gamma}_{2,2,n}^l)^{\top} & \tilde{\Gamma}_{2,3,n}^l
\end{pmatrix*}(x,y).
\]

The following two lemmas can be proved in the same manner as Lemmas~\ref{lem:1} and~\ref{lem:2}, using Proposition~3.16 in Ogihara~\cite{ogi23} and the boundedness of \(\partial_\theta^3 \mathcal{Y}_2\) following from (A1). Therefore, the proofs are omitted.

\begin{lem}\label{lem:6}
Assume conditions (A1) -- (A6). Then, under \(H_{0}^{(\theta)}\), 
$\tilde{\Gamma}_{2,n}^1(\tilde{\theta}_n,\hat{\theta}_n)
\overset{P}{\rightarrow}\Gamma_{2}$ as $n\to\infty$.
\end{lem}

\begin{lem}\label{lem:7}
Assume conditions (A1) -- (A6). Then, under \(H_{0}^{(\theta)}\), 
$\tilde{\Gamma}_{2,n}^0 (\tilde{\theta}_{n}, \theta_{0})
\overset{P}{\rightarrow}\Gamma_{2}$, and $\tilde{\Gamma}_{2,n}^0 (\tilde{\theta}_{n}, \hat{\theta}_{n})
\overset{P}{\rightarrow}\Gamma_{2}$ as $n\to\infty$.
\end{lem}

For \(1\leq r_2<d_2\), define the projection \(\tilde{P}_{0}:\mathbb{R}^{d_{2}} \to \mathbb{R}^{d_{2}-r_{2}}\) onto the last \(d_{2}-r_{2}\) components by
\[
\tilde{P}_{0}
\begin{pmatrix}
\theta^{(1)} \\ \theta^{(2)} \\ \vdots \\ \theta^{(d_{2})}
\end{pmatrix}
=
\begin{pmatrix}
\theta^{(r_{2}+1)} \\ \vdots \\ \theta^{(d_{2})}
\end{pmatrix}.
\]

\begin{lem}\label{lem:8}
Assume \(1\leq r_2<d_2\) and conditions (A1) -- (A6). Under \(H_{0}^{(\theta)}\), provided that \(\tilde{\theta}_n\in \Theta_{2}\), we have
\[
\frac{1}{\sqrt{nh_n}}\partial_{\tilde{P}_{0}\theta}\mathbb{H}_{n}^{2}(\theta_{0})
=\tilde{\Gamma}_{2,3,n}^0 (\tilde{\theta}_{n}, \theta_{0})\,\sqrt{nh_n}\,\tilde{P}_{0}\bigl(\tilde{\theta}_{n}-\theta_{0}\bigr).
\]
\begin{proof}
Under \(H_{0}^{(\theta)}\), note that the first \(r_2\) components of \(\tilde{\theta}_n\) and \(\theta_{0}\) coincide. Expanding in \(\tilde{P}_{0}\theta\) yields
\begin{equation}\label{eq:6}
\frac{1}{\sqrt{nh_n}}\partial_{\tilde{P}_{0}\theta}\mathbb{H}_{n}^{2}(\tilde{\theta}_{n})
-\frac{1}{\sqrt{nh_n}}\partial_{\tilde{P}_{0}\theta}\mathbb{H}_{n}^{2}(\theta_{0})
=-\tilde{\Gamma}_{2,3,n}^0 (\tilde{\theta}_{n},\theta_{0})\,\sqrt{nh_n}\,\tilde{P}_{0}\bigl(\tilde{\theta}_{n}-\theta_{0}\bigr).
\end{equation}
Since \(\tilde{\theta}_{n}\in\Theta_{2}\) implies \(\partial_{\tilde{P}_{0}\theta}\mathbb{H}_{n}^{2}(\tilde{\theta}_{n})=0\), the conclusion follows.
\end{proof}
\end{lem}

\begin{lem}\label{lem:9}
Assume \(1\le r_2<d_2\) and let \(Z\sim\mathcal{N}(0, \mathcal{E}_{d_{2}})\). Under conditions (A1) -- (A6) and \(H_{0}^{(\theta)}\),
\[
\sqrt{nh_{n}}\bigl(\tilde{\theta}_{n} - \hat{\theta}_{n}\bigr)
\overset{d}{\rightarrow} -\Gamma_{2}^{-1}\Bigl(\mathcal{E}_{d_{2}}-\Gamma_{2}\Lambda_{2}\Bigr)\Gamma_{2}^{1/2}Z, \quad (n\to\infty).
\]
\begin{proof}
A Taylor expansion in \(\theta\) yields
\begin{equation}\label{lem:9-eq1}
\begin{split}
\frac{1}{\sqrt{nh_{n}}}\partial_{\theta}\mathbb{H}_{n}^{2}(\tilde{\theta}_{n})
-\frac{1}{\sqrt{nh_{n}}}\partial_{\theta}\mathbb{H}_{n}^{2}(\theta_{0})
=-\tilde{\Gamma}_{2,n}^0 (\tilde{\theta}_{n},\theta_{0})\,\sqrt{nh_{n}}\bigl(\tilde{\theta}_{n}-\theta_{0}\bigr).
\end{split}
\end{equation}
Define
\[
\tilde{\Lambda}_{2,n}(\tilde{\theta}_{n},\theta_{0}) =
\begin{pmatrix}
O & O \\[2mm]
O & (\tilde{\Gamma}_{2,3,n}^0)^{-1}(\tilde{\theta}_{n},\theta_{0})
\end{pmatrix}\mathbf{1}_{\{\det \tilde{\Gamma}_{2,3,n}^0 (\tilde{\theta}_{n},\theta_{0})>0\}}.
\]
Then, by Lemma~\ref{lem:8},
\[
\sqrt{nh_{n}}\bigl(\tilde{\theta}_{n}-\theta_{0}\bigr)
=\tilde{\Lambda}_{2,n}(\tilde{\theta}_{n},\theta_{0})\frac{1}{\sqrt{nh_{n}}}\partial_{\theta}\mathbb{H}_{n}^{2}(\theta_{0})
\quad \text{on } \{\det \tilde{\Gamma}_{2,3,n}^0 (\tilde{\theta}_{n},\theta_{0})>0\}.
\]
Thus, combining with \eqref{lem:9-eq1},
\[
\frac{1}{\sqrt{nh_{n}}}\partial_{\theta}\mathbb{H}_{n}^{2}(\tilde{\theta}_{n})
-\frac{1}{\sqrt{nh_{n}}}\partial_{\theta}\mathbb{H}_{n}^{2}(\theta_{0})
=-\tilde{\Gamma}_{2,n}^0 (\tilde{\theta}_{n},\theta_{0})
\tilde{\Lambda}_{2,n} (\tilde{\theta}_{n},\theta_{0})\frac{1}{\sqrt{nh_{n}}}\partial_{\theta}\mathbb{H}_{n}^{2}(\theta_{0})
\]
on \(\{\det \tilde{\Gamma}_{2,3,n}^0 (\tilde{\theta}_{n},\theta_{0})>0\}\). Hence,
\[
\frac{1}{\sqrt{nh_{n}}}\partial_{\theta}\mathbb{H}_{n}^{2}(\tilde{\theta}_{n})
=\Bigl(\mathcal{E}_{d_{2}}-\tilde{\Gamma}_{2,n}^0 (\tilde{\theta}_{n},\theta_{0})
\tilde{\Lambda}_{2,n}(\tilde{\theta}_{n},\theta_{0})\Bigr)\frac{1}{\sqrt{nh_{n}}}\partial_{\theta}\mathbb{H}_{n}^{2}(\theta_{0})
\quad \text{on } \{\det \tilde{\Gamma}_{2,3,n}^0 (\tilde{\theta}_{n},\theta_{0})>0\}.
\]
By Lemma~\ref{lem:7} and the positive definiteness of \(\Gamma_{2}\), we have
$P\Bigl(\det \tilde{\Gamma}_{2,3,n}^0 (\tilde{\theta}_{n},\theta_{0})\le 0\Bigr)\to 0 $,
which implies
$\tilde{\Lambda}_{2,n}(\tilde{\theta}_{n},\theta_{0})\overset{P}{\rightarrow}\Lambda_{2}$.
Furthermore, by (3.59) in Ogihara~\cite{ogi23},
$(nh_{n})^{-1/2}\partial_{\theta}\mathbb{H}_{n}^{2}(\theta_{0})
\overset{d}{\rightarrow}\mathcal{N}(0,\Gamma_{2})$.
Thus, letting \(Z\sim\mathcal{N}(0,\mathcal{E}_{d_{2}})\),
\begin{equation}\label{eq:7}
\frac{1}{\sqrt{nh_{n}}}\partial_{\theta}\mathbb{H}_{n}^{2}(\tilde{\theta}_{n})
\overset{d}{\rightarrow} (\mathcal{E}_{d_{2}}-\Gamma_{2}\Lambda_{2})\Gamma_{2}^{1/2}Z.
\end{equation}
On the other hand, by the consistency of \(\hat{\theta}_{n}\) (so that for sufficiently large \(n\), \(\hat{\theta}_{n}\in\Theta_{2}\)) we have \(\partial_{\theta}\mathbb{H}_{n}^{2}(\hat{\theta}_{n})=0\). Then, by a Taylor expansion,
\begin{align*}
\frac{1}{\sqrt{nh_{n}}}\partial_{\theta}\mathbb{H}_{n}^{2}(\tilde{\theta}_{n})
-\frac{1}{\sqrt{nh_{n}}}\partial_{\theta}\mathbb{H}_{n}^{2}(\hat{\theta}_{n})
=-\tilde{\Gamma}_{2,n}^0 (\tilde{\theta}_{n},\hat{\theta}_{n})\,\sqrt{nh_{n}}\bigl(\tilde{\theta}_{n}-\hat{\theta}_{n}\bigr).
\end{align*}
Hence,
\begin{equation}\label{eq:8}
\sqrt{nh_{n}}\bigl(\tilde{\theta}_{n}-\hat{\theta}_{n}\bigr)
=-(\tilde{\Gamma}_{2,n}^0)^{-1}(\tilde{\theta}_{n},\hat{\theta}_{n})\frac{1}{\sqrt{nh_{n}}}\partial_{\theta}\mathbb{H}_{n}^{2}(\tilde{\theta}_{n})
\quad \text{on } \{\det \tilde{\Gamma}_{2,n}^0 (\tilde{\theta}_{n},\hat{\theta}_{n})>0\}.
\end{equation}
Together with Lemma~\ref{lem:7}, \eqref{eq:7} and \eqref{eq:8}, we have
\[
\sqrt{nh_{n}}\bigl(\tilde{\theta}_{n}-\hat{\theta}_{n}\bigr)
\overset{d}{\rightarrow} -\Gamma_{2}^{-1}\Bigl(\mathcal{E}_{d_{2}}-\Gamma_{2}\Lambda_{2}\Bigr)\Gamma_{2}^{1/2}Z, \quad (n\to\infty).
\]
\end{proof}
\end{lem}

\subsubsection*{Proof of Theorem \ref{thm:theta0}}
By Taylor's theorem,
\begin{align*}
\mathbb{H}_{n}^{2}(\tilde{\theta}_{n})
&=\mathbb{H}_{n}^{2}(\hat{\theta}_{n})
+\partial_{\theta}\mathbb{H}_{n}^{2}(\hat{\theta}_{n})\bigl(\tilde{\theta}_{n}-\hat{\theta}_{n}\bigr) 
-\frac{1}{2}\tilde{\Gamma}_{2,n}^1(\tilde{\theta}_n, \hat{\theta}_n)
\Bigl\llbracket\Bigl[\sqrt{nh_{n}}\bigl(\tilde{\theta}_{n}-\hat{\theta}_{n}\bigr)\Bigr]^{\otimes 2}\Bigr\rrbracket.
\end{align*}
By the consistency of \(\hat{\theta}_{n}\) (i.e., \(P(\hat{\theta}_{n}\in\Theta_{2})\to1\)), we have \(\partial_{\theta}\mathbb{H}_{n}^{2}(\hat{\theta}_{n})=0\). Thus,
\begin{equation}\label{eq:9}
\mathbb{H}_{n}^{2}(\tilde{\theta}_{n})-\mathbb{H}_{n}^{2}(\hat{\theta}_{n})
=-\frac{1}{2}\tilde{\Gamma}_{2,n}^1(\tilde{\theta}_n, \hat{\theta}_n)
\Bigl\llbracket\Bigl[\sqrt{nh_{n}}\bigl(\tilde{\theta}_{n}-\hat{\theta}_{n}\bigr)\Bigr]^{\otimes 2}\Bigr\rrbracket.
\end{equation}
First, consider the case \(r_{2}=d_{2}\). Then \(\tilde{\theta}_{n}=\theta_{0}=0\). Let \(Z\sim\mathcal{N}(0,\mathcal{E}_{d_{2}})\). By Ogihara~\cite{ogi23} Theorem~2.3,
$\sqrt{nh_{n}}(\hat{\theta}_{n}-\theta_{0})
\overset{d}{\rightarrow}\Gamma_{2}^{-1/2}Z$,
so that
\begin{equation}\label{eq:10}
\sqrt{nh_{n}}(\tilde{\theta}_{n}-\hat{\theta}_{n})
\overset{d}{\rightarrow}\Gamma_{2}^{-1/2}Z.
\end{equation}
Then, by \eqref{eq:9}, \eqref{eq:10}, Lemma~\ref{lem:6}, and Slutsky\textquotesingle s theorem,
\[
T_{n}^{2}=2\Bigl(\mathbb{H}_{n}^{2}(\hat{\theta}_{n})-\mathbb{H}_{n}^{2}(\tilde{\theta}_{n})\Bigr)
\overset{d}{\rightarrow} Z^{\top}\Gamma_{2}^{-1/2}\Gamma_{2}\Gamma_{2}^{-1/2}Z
\sim \chi^{2}_{d_{2}}=\chi^{2}_{r_{2}}.
\]

Next, consider the case \(1\le r_{2} < d_{2}\). Using \eqref{eq:9} together with Lemmas~\ref{lem:6} and \ref{lem:9} and applying Slutsky\textquotesingle s theorem, we obtain
\begin{align*}
T_{n}^{2} &= -2\Bigl(\mathbb{H}_{n}^{2}(\tilde{\theta}_{n})-\mathbb{H}_{n}^{2}(\hat{\theta}_{n})\Bigr)\\[1mm]
&=\tilde{\Gamma}_{2,n}^1(\tilde{\theta}_n, \hat{\theta}_n)
\Bigl\llbracket\Bigl[\sqrt{nh_{n}}\bigl(\tilde{\theta}_{n}-\hat{\theta}_{n}\bigr)\Bigr]^{\otimes 2}\Bigr\rrbracket\\[1mm]
&\overset{d}{\rightarrow} Z^{\top}\Gamma^{1/2}_{2}\Bigl(\mathcal{E}_{d_{2}}-\Gamma_{2}\Lambda_{2}\Bigr)^{\top}\Gamma_{2}^{-1}\Bigl(\mathcal{E}_{d_{2}}-\Gamma_{2}\Lambda_{2}\Bigr)\Gamma^{1/2}_{2}Z\\[1mm]
&= Z^{\top}\Gamma^{1/2}_{2}\Bigl(\Gamma_{2}^{-1}-\Lambda_{2}\Bigr)\Gamma^{1/2}_{2}Z.
\end{align*}
Note that
\[
\Lambda_{2}\Gamma_{2}\Lambda_{2}=
\begin{pmatrix}
O & O \\[2mm]
O & (\Gamma_{2,3})^{-1}
\end{pmatrix}
\begin{pmatrix*}[c]
\Gamma_{2,1} & \Gamma_{2,2} \\[2mm]
\Gamma_{2,2}^{\top} & \Gamma_{2,3}
\end{pmatrix*}
\begin{pmatrix}
O & O \\[2mm]
O & (\Gamma_{2,3})^{-1}
\end{pmatrix}
=\Lambda_{2}.
\]
Also, since
\begin{align*}
\Gamma_{2}^{1/2}\Bigl(\Gamma_{2}^{-1}-\Lambda_{2}\Bigr)\Gamma_{2}^{1/2}\Gamma_{2}^{1/2}\Bigl(\Gamma_{2}^{-1}-\Lambda_{2}\Bigr)\Gamma_{2}^{1/2}
&=\Gamma_{2}^{1/2}\Bigl(\Gamma_{2}^{-1}-\Lambda_{2}\Bigr)(\mathcal{E}_{d_{2}}-\Gamma_{2}\Lambda_{2})\Gamma_{2}^{1/2}\\[1mm]
&=\Gamma_{2}^{1/2}\Bigl(\Gamma_{2}^{-1}-\Lambda_{2}\Bigr)\Gamma_{2}^{1/2},
\end{align*}
it follows that \(\Gamma_{2}^{1/2}(\Gamma_{2}^{-1}-\Lambda_{2})\Gamma_{2}^{1/2}\) is a projection matrix whose rank equals its trace:
\begin{align*}
\mathrm{tr}\Bigl\{\Gamma_{2}^{1/2}(\Gamma_{2}^{-1}-\Lambda_{2})\Gamma_{2}^{1/2}\Bigr\}
=\mathrm{tr}\Bigl(\mathcal{E}_{d_{2}}-\Gamma_{2}\Lambda_{2}\Bigr)
=\mathrm{tr}\begin{pmatrix*}[c]
\mathcal{E}_{r_{2}} & * \\[2mm]
O & O
\end{pmatrix*}
=r_{2},
\end{align*}
which leads to the conclusion.

\qed

\begin{lem}\label{lem:10}
Assume (A1) -- (A6). Then,
\[
\sup_{\theta\in\Theta_{2}}\left|\frac{1}{nh_{n}}\Bigl(\mathbb{H}_{n}^{2}(\theta)-\mathbb{H}_{n}^{2}(\hat{\theta}_{n})\Bigr)-\mathcal{Y}_{2}(\theta)\right|
\overset{P}{\rightarrow}0, \quad (n\to\infty).
\]
\begin{proof}
We write
\begin{align*}
&\left|\frac{1}{nh_{n}}\Bigl(\mathbb{H}_{n}^{2}(\theta)-\mathbb{H}_{n}^{2}(\hat{\theta}_{n})\Bigr)-\mathcal{Y}_{2}(\theta)\right|\\[1mm]
&\quad=\Biggl|\frac{1}{nh_{n}}\Bigl(\mathbb{H}_{n}^{2}(\theta)-\mathbb{H}_{n}^{2}(\theta_{0})\Bigr)-\mathcal{Y}_{2}(\theta)
-\frac{1}{nh_{n}}\Bigl(\mathbb{H}_{n}^{2}(\hat{\theta}_{n})-\mathbb{H}_{n}^{2}(\theta_{0})\Bigr)
+\mathcal{Y}_{2}(\hat{\theta}_{n})-\mathcal{Y}_{2}(\hat{\theta}_{n})\Biggr|\\[1mm]
&\quad\le \left|\frac{1}{nh_{n}}\Bigl(\mathbb{H}_{n}^{2}(\theta)-\mathbb{H}_{n}^{2}(\theta_{0})\Bigr)-\mathcal{Y}_{2}(\theta)\right|
+\left|\frac{1}{nh_{n}}\Bigl(\mathbb{H}_{n}^{2}(\hat{\theta}_{n})-\mathbb{H}_{n}^{2}(\theta_{0})\Bigr)-\mathcal{Y}_{2}(\hat{\theta}_{n})\right|
+\left|\mathcal{Y}_{2}(\hat{\theta}_{n})\right|.
\end{align*}
By Ogihara~\cite{ogi23} Proposition~3.16,
\[
\sup_{\theta\in\Theta_{2}}\left|\frac{1}{nh_{n}}\Bigl(\mathbb{H}_{n}^{2}(\theta)-\mathbb{H}_{n}^{2}(\theta_{0})\Bigr)-\mathcal{Y}_{2}(\theta)\right|
\overset{P}{\rightarrow}0.
\]
Thus, it suffices to show that
$-\mathcal{Y}_{2}(\hat{\theta}_{n})\overset{P}{\rightarrow}0$.
Noting that \(\mathcal{Y}_{2}(\theta_{0})=0\) and using the boundedness of the derivatives of \(\mathcal{Y}_{2}\) following from (A1) together with the consistency of \(\hat{\theta}_n\), we have
\begin{align*}
0\le -\mathcal{Y}_{2}(\hat{\theta}_{n})
&=\left|\mathcal{Y}_{2}(\hat{\theta}_{n})-\mathcal{Y}_{2}(\theta_{0})\right|\\[1mm]
&=\left|\left(\int_{0}^{1}\partial_{\theta}\mathcal{Y}_{2}\Bigl(\theta_{0}+u\bigl(\hat{\theta}_{n}-\theta_{0}\bigr)\Bigr)\mathrm{d}u\right)
\bigl(\hat{\theta}_{n}-\theta_{0}\bigr)\right|\\[1mm]
&\le K''\lvert\hat{\theta}_{n}-\theta_{0}\rvert
\overset{P}{\rightarrow}0.
\end{align*}
This completes the proof.
\end{proof}
\end{lem}

\subsubsection*{Proof of Theorem \ref{thm:theta1}}
By Ogihara~\cite{ogi23} (3.56) and (A6), there exists a positive constant \(c\) such that
$\mathcal{Y}_{2}(\theta) \le -c\,\lvert\theta-\theta_{0}\rvert^{2}$.
Under \(H_{1}^{(\theta)}\) there exists some \(\ell_{2}\in\{1,\dots,r_{2}\}\) for which \(\theta_{0}^{(\ell_{2})}\neq 0\). Hence,
\[
\frac{1}{2nh_{n}}T_{n}^{2} + \left|\frac{1}{2nh_{n}}T_{n}^{2} + \mathcal{Y}_{2}(\tilde{\theta}_{n})\right|
\ge -\mathcal{Y}_{2}(\tilde{\theta}_{n})
\ge c\,\lvert\tilde{\theta}_{n}-\theta_{0}\rvert^{2}
\ge c\,\lvert\theta_{0}^{(\ell_{2})}\rvert^{2} > 0.
\]
Thus,
\begin{align*}
P\bigl(T_{n}^{2} \le M\bigr)
&= P\Biggl(\frac{1}{2nh_{n}}T_{n}^{2} + \left|\frac{1}{2nh_{n}}T_{n}^{2} + \mathcal{Y}_{2}(\tilde{\theta}_{n})\right|
\le \frac{M}{2nh_{n}} + \left|\frac{1}{2nh_{n}}T_{n}^{2} + \mathcal{Y}_{2}(\tilde{\theta}_{n})\right|\Biggr)\\[1mm]
&\le P\Biggl(c\,\lvert\theta_{0}^{(\ell_{2})}\rvert^{2}
\le \frac{M}{2nh_{n}} + \Biggl|\frac{1}{nh_{n}}\Bigl(\mathbb{H}_{n}^{2}(\hat{\theta}_{n}) - \mathbb{H}_{n}^{2}(\tilde{\theta}_{n})\Bigr) + \mathcal{Y}_{2}(\tilde{\theta}_{n})\Biggr|\Biggr)\\[1mm]
&\le P\Biggl(c\,\lvert\theta_{0}^{(\ell_{2})}\rvert^{2} - \frac{M}{2nh_{n}}
\le \sup_{\theta\in\Theta_{2}}\Biggl|\frac{1}{nh_{n}}\Bigl(\mathbb{H}_{n}^{2}(\hat{\theta}_{n})-\mathbb{H}_{n}^{2}(\theta)\Bigr) + \mathcal{Y}_{2}(\theta)\Biggr|\Biggr)\\[1mm]
&\le P\Biggl(\frac{1}{2}c\,\lvert\theta_{0}^{(\ell_{2})}\rvert^{2}
\le \sup_{\theta\in\Theta_{2}}\Biggl|\frac{1}{nh_{n}}\Bigl(\mathbb{H}_{n}^{2}(\theta)-\mathbb{H}_{n}^{2}(\hat{\theta}_{n})\Bigr) - \mathcal{Y}_{2}(\theta)\Biggr|\Biggr) + o(1)\\[1mm]
&\to 0 \quad \text{(by Lemma \ref{lem:10}).}
\end{align*}

\qed

\subsection{Proof of Theorem~\ref{aumpi-thm}}

Our model’s true log-likelihood function, excluding constant terms, can be written as
\[
    \mathbb{H}_{n}(\sigma, \theta) = -\frac{1}{2}\bar{X}(\theta)^{\top}S_{n}^{-1}(\sigma)\bar{X}(\theta) -\frac{1}{2}\log\det S_{n} (\sigma).
\]

Define
\[
    \hat{\sigma}_{n}^{\ast}\in \operatorname{argmax}_{\sigma\in\bar{\Theta}_{1}}\mathbb{H}_{n}(\sigma, \theta_{0}), \quad
    \tilde{\sigma}_{n}^{\ast}\in \operatorname{argmax}_{\sigma\in\bar{\Theta}_{0, 1}}\mathbb{H}_{n}(\sigma, \theta_{0}),
\]
\[
    \hat{\theta}_{n}^{\ast}\in \operatorname{argmax}_{\theta\in\bar{\Theta}_{2}}\mathbb{H}_{n}(\sigma_0, \theta), \quad
    \tilde{\theta}_{n}^{\ast}\in \operatorname{argmax}_{\theta\in\bar{\Theta}_{0, 2}}\mathbb{H}_{n}(\sigma_0, \theta),
\]
and set
\[
    \mathrm{LR}_{1} = 2\Bigl(\mathbb{H}_{n}(\hat{\sigma}_{n}^{\ast}, \theta_{0}) - \mathbb{H}_{n}(\tilde{\sigma}_{n}^{\ast}, \theta_{0})\Bigr),
    \quad \mathrm{LR}_{2} = 2\Bigl(\mathbb{H}_{n}(\sigma_0,\hat{\theta}_{n}^{\ast}) - \mathbb{H}_{n}(\sigma_0,\tilde{\theta}_{n}^{\ast})\Bigr).
\]
Then, $\mathrm{LR}_{1}$ and $\mathrm{LR}_{2}$ are the likelihood ratio test statistics.

In the construction of \(T_n^1\) and \(T_n^2\), the quasi-log-likelihood functions \(\mathbb{H}_n^1\) and \(\mathbb{H}_n^2\) for the adaptive estimation of \(\sigma\) and \(\theta\) are employed. The following lemma indicates that the discrepancy between the true likelihood and the quasi-log-likelihood functions \(\mathbb{H}_n^1\) and \(\mathbb{H}_n^2\) is asymptotically negligible.
\begin{lem}\label{lem:ex-1}
Assume conditions (A1)--(A6). For each \(k \in \{0,1,2,3\}\),
\[
    \sup_{\sigma\in\Theta_{1}}\left| \frac{1}{\sqrt{n}}\Bigl(\partial_{\sigma}^{k}\mathbb{H}_{n}(\sigma, \theta_{0}) - \partial_{\sigma}^{k}\mathbb{H}_{n}^{1}(\sigma)\Bigr) \right| \overset{P}{\rightarrow} 0,
\]
\[
    \sup_{\theta\in\Theta_{2}}\left| \frac{1}{\sqrt{nh_n}}\Bigl(\partial_{\theta}^{k}(\mathbb{H}_{n}(\sigma_0, \theta)-\mathbb{H}_{n}(\sigma_0, \theta_0)) - \partial_{\theta}^{k}(\mathbb{H}_{n}^{2}(\theta)-\mathbb{H}_{n}^{2}(\theta_0))\Bigr) \right| \overset{P}{\rightarrow} 0
\]
as $n\to\infty$.
\begin{proof}
Define 
\[
    X_{t}^{c} = \int_{0}^{t} b_{s}(\sigma_{0})\,\mathrm{d}W_{s}.
\]
Since \(\bar{X}(\theta_{0}) = \Delta X^{c}\) and \(\Delta X = \Delta X^{c} + \Delta V(\theta_{0})\), we have
\begin{align*}
    \mathbb{H}_{n}(\sigma, \theta_{0}) - \mathbb{H}_{n}^{1}(\sigma)
    &= -\frac{1}{2}\bar{X}(\theta_{0})^{\top}S_{n}^{-1}(\sigma)\bar{X}(\theta_{0})
    + \frac{1}{2}\Delta X^{\top}S_{n}^{-1}(\sigma)\Delta X\\[1mm]
    &= \Delta V(\theta_{0})^{\top}S_{n}^{-1}(\sigma)\Delta X^{c}
    + \frac{1}{2}\Delta V(\theta_{0})^{\top}S_{n}^{-1}(\sigma)\Delta V(\theta_{0}).
\end{align*}
Then, by the argument in Ogihara~\cite{ogi23} (Lemma 3.6),
\[
    \frac{1}{\sqrt{n}} \sup_{\sigma}\left|\Delta V(\theta_{0})^{\top}S_{n}^{-1}(\sigma)\Delta X^{c}\right| \overset{P}{\rightarrow} 0,
\]
and
\[
    \frac{1}{\sqrt{n}} \sup_{\sigma}\left|\Delta V(\theta_{0})^{\top}S_{n}^{-1}(\sigma)\Delta V(\theta_{0})\right| \overset{P}{\rightarrow} 0.
\]
Hence, the first result for \(k=0\) follows. The same argument applies for \(k\in\{1,2,3\}\).

For the second result, we have
\begin{equation}\label{ulan-eq1}
\mathbb{H}_n(\sigma_0,\theta)-\mathbb{H}_n^2(\theta)+\frac{1}{2}\log\det S_n(\sigma_0) 
=-\frac{1}{2}\bar{X}(\theta)^\top (S_n^{-1}(\sigma_0)-S_n^{-1}(\hat{\sigma}_n))\bar{X}(\theta).
\end{equation}
Let $S_{n,0}=S_n(\sigma_0)$. Then, a simple calculation shows
$$\bar{X}(\theta)^\top S_{n,0}^{-1}\bar{X}(\theta) = \Delta X^{\top} S_{n,0}^{-1} \Delta X 
-2\Delta V(\theta)^\top S_{n,0}^{-1}\Delta X^c - \Delta V(\theta)^\top S_{n,0}^{-1}(2\Delta V(\theta_0)-\Delta V(\theta)).$$
Moreover, (3.43) in Ogihara~\cite{ogi23} yields
$$\bar{X}(\theta)^\top S_n^{-1}(\hat{\sigma}_n)\bar{X}(\theta) = \Delta X^{\top} S_n^{-1}(\hat{\sigma}_n) \Delta X 
-2\Delta V(\theta)^\top S_{n,0}^{-1}\Delta X^c - \Delta V(\theta)^\top S_{n,0}^{-1}(2\Delta V(\theta_0)-\Delta V(\theta))+o_p(\sqrt{nh_n}).$$
Therefore, the right-hand side of (\ref{ulan-eq1}) is equal to
$$\Delta X^{\top}(S_{n,0}^{-1}-S_n^{-1}(\hat{\sigma}_n))\Delta X + o_p(\sqrt{nh_n}),$$
which implies the second result.
\end{proof}
\end{lem}

\begin{lem}\label{lem:ex-2}
Assume conditions (A1)--(A4) and (A6). Then, for each \(k \in \{0,1,2,3\}\),
\[
    \sup_{\sigma\in\Theta_{1}}\left| \frac{1}{n}\partial_{\sigma}^{k}\Bigl(\mathbb{H}_{n}(\sigma, \theta_{0}) - \mathbb{H}_{n}(\sigma_{0}, \theta_{0})\Bigr) - \partial_{\sigma}^{k}\mathcal{Y}_{1}(\sigma)\right| \overset{P}{\rightarrow} 0,
    \quad \sup_{\theta\in\Theta_{2}}\left| \frac{1}{nh_n}\partial_{\theta}^{k}\Bigl(\mathbb{H}_{n}(\sigma_0, \theta) - \mathbb{H}_{n}(\sigma_{0}, \theta_{0})\Bigr) - \partial_{\theta}^{k}\mathcal{Y}_{2}(\theta)\right| \overset{P}{\rightarrow} 0
\]
as $n\to\infty$.
\begin{proof}
Note that
\begin{align*}
    &\left|\frac{1}{n}\Bigl(\mathbb{H}_{n}(\sigma, \theta_{0}) - \mathbb{H}_{n}(\sigma_{0}, \theta_{0})\Bigr) - \mathcal{Y}_{1}(\sigma)\right|\\[1mm]
    &\quad \leq \left|\frac{1}{n}\Bigl(\mathbb{H}_{n}^{1}(\sigma) - \mathbb{H}_{n}^{1}(\sigma_{0})\Bigr) - \mathcal{Y}_{1}(\sigma)\right|
    + \left|\frac{1}{n}\Bigl(\mathbb{H}_{n}(\sigma, \theta_{0}) - \mathbb{H}_{n}^{1}(\sigma)\Bigr)\right|
    + \left|\frac{1}{n}\Bigl(\mathbb{H}_{n}(\sigma_{0}, \theta_{0}) - \mathbb{H}_{n}^{1}(\sigma_{0})\Bigr)\right|.
\end{align*}
By Ogihara~\cite{ogi23} (Proposition 3.8), the first term converges to zero in probability. In addition, by Lemma \ref{lem:ex-1} the second and third terms also converge to zero, yielding the first result for \(k=0\). The same argument applies for \(k\in\{1,2,3\}\), and the second result.
\end{proof}
\end{lem}

Theorem 3 in Hall and Mathiason~\cite{hal-mat90} and Theorem 3 in Choi et al.~\cite{cho-etal96} together with the following uniform LAN properties yield that the tests constructed by ${\rm LR}_1$ and ${\rm LR}_2$ are AUMPI tests.

\begin{prop}\label{ulan-prop}
Assume conditions (A1)--(A6). For $u_1\in\mathbb{R}^{d_1}$ and $u_2\in\mathbb{R}^{d_2}$, define
\[
    \sigma_{n}(u_1)=\sigma_{0}+\Gamma_{1}^{-1/2}\frac{u_1}{\sqrt{n}},
    \quad \theta_n(u_2)=\theta_0+\Gamma_2^{-1/2}\frac{u_2}{\sqrt{nh_n}}.
\]
Then there exists a sequence of $d_{1}$- and $d_2$-dimensional random variables $(\Delta_{n}^{1})_{n\geq 1}$ and $(\Delta_{n}^{2})_{n\geq 1}$, respectively, such that for any $M>0$
\[
    \sup_{|u_1|\leq M}\left|\mathbb{H}_{n}(\sigma_{n}(u_1), \theta_{0}) - \mathbb{H}_{n}(\sigma_{0}, \theta_{0}) - \Bigl(u_1^{\top}\Delta_{n}^{1} - \frac{1}{2}|u_1|^2\Bigr)\right| \overset{P}{\rightarrow} 0,
\]
\[
    \sup_{|u_2|\leq M}\left|\mathbb{H}_{n}(\sigma_0, \theta_{n}(u_2)) - \mathbb{H}_{n}(\sigma_{0}, \theta_{0}) - \Bigl(u_2^{\top}\Delta_{n}^{2} - \frac{1}{2}|u_2|^2\Bigr)\right| \overset{P}{\rightarrow} 0,
\]
and
\[
    \Delta_{n}^{1} \overset{d}{\rightarrow} \mathcal{N}(0, \mathcal{E}_{d_{1}}), 
    \quad \Delta_{n}^{2} \overset{d}{\rightarrow} \mathcal{N}(0, \mathcal{E}_{d_{2}})
\]
as $n\to\infty$.
\begin{proof}
Taylor's theorem yields
\begin{align*}
    \mathbb{H}_{n}(\sigma_{n}(u), \theta_{0}) - \mathbb{H}_{n}(\sigma_{0}, \theta_{0})
    &= u^{\top}\Gamma_{1}^{-1/2}\int_{0}^{1} \partial_{\sigma}\mathbb{H}_{n}(\sigma_{n}(tu), \theta_{0}) \, \mathrm{d}t \\
    &= \frac{1}{\sqrt{n}}u^{\top}\Gamma_{1}^{-1/2}\partial_{\sigma}\mathbb{H}_{n}(\sigma_{0}, \theta_{0}) + \frac{1}{2}u^{\top}\Gamma_{1}^{-1/2}\frac{1}{n}\partial_{\sigma}^{2}\mathbb{H}_{n}(\sigma_{0}, \theta_{0})\Gamma_{1}^{-1/2}u \\
    &\quad +\frac{1}{\sqrt{n}}\sum_{i,j,k}\int_{0}^{1}\frac{(1-s)^{2}}{2}\frac{1}{n}\partial_{\sigma^{(i)}}\partial_{\sigma^{(j)}}\partial_{\sigma^{(k)}} \mathbb{H}_{n}(\sigma_{n}(su), \theta_{0})\,\mathrm{d}s\,[\Gamma_{1}^{-1/2}u]^{(i)}[\Gamma_{1}^{-1/2}u]^{(j)}[\Gamma_{1}^{-1/2}u]^{(k)}.
\end{align*}
By Lemma \ref{lem:ex-2},
\[
    \sup_{\sigma\in\Theta_{1}}\left|
    \frac{1}{n}\partial_{\sigma}^{3}\mathbb{H}_{n}(\sigma, \theta_{0}) - \partial_{\sigma}^{3}\mathcal{Y}_{1}(\sigma)
    \right|
    \overset{P}{\rightarrow}0.
\]
Since $\partial_{\sigma}^{3}\mathcal{Y}_{1}(\sigma)$ is bounded by (A1), it follows that
\[
    \left|\sum_{i,j,k}\int_{0}^{1}\frac{(1-s)^{2}}{2}\frac{1}{n}\partial_{\sigma^{(i)}}\partial_{\sigma^{(j)}}\partial_{\sigma^{(k)}}\mathbb{H}_{n}(\sigma_{n}(su), \theta_{0})\,\mathrm{d}s\right| = O_{p}(1).
\]
For $\|u\|\leq M$, each component of $\Gamma_{1}^{-1/2}u$ is bounded, so that
\[
    \sup_{\|u\|\leq M} \left|\frac{1}{\sqrt{n}}\sum_{i,j,k}\int_{0}^{1}\frac{(1-s)^{2}}{2}\frac{1}{n}\partial_{\sigma^{(i)}}\partial_{\sigma^{(j)}}\partial_{\sigma^{(k)}}\mathbb{H}_{n}(\sigma_{n}(su), \theta_{0})\,\mathrm{d}s\,[\Gamma_{1}^{-1/2}u]^{(i)}[\Gamma_{1}^{-1/2}u]^{(j)}[\Gamma_{1}^{-1/2}u]^{(k)}\right| \overset{P}{\rightarrow} 0.
\]
Moreover, by Lemma \ref{lem:ex-1}, Lemma \ref{lem:ex-2}, and Proposition 3.14 in Ogihara~\cite{ogi23},
\[
    -\frac{1}{n}\partial_{\sigma}^{2}\mathbb{H}_{n}(\sigma_{0}, \theta_{0}) \overset{P}{\rightarrow} \Gamma_{1},
\]
\[
    \frac{1}{\sqrt{n}}\partial_{\sigma}\mathbb{H}_{n}(\sigma_{0}, \theta_{0}) \overset{d}{\rightarrow} \mathcal{N}(0, \Gamma_{1}).
\]
Thus, the results for $\mathbb{H}_{n}(\sigma_{n}(u_1), \theta_{0}) - \mathbb{H}_{n}(\sigma_{0}, \theta_{0})$ follows.
The results for $\mathbb{H}_{n}(\sigma_0, \theta_{n}(u_2)) - \mathbb{H}_{n}(\sigma_{0}, \theta_{0})$ follows by a similar argument with (3.59) in Ogihara~\cite{ogi23}.
\end{proof}
\end{prop}

Since the test constructed by $\mathrm{LR}_1$ is AUMPI, it suffices to show that $\mathrm{LR}_1 - T_n^1 \to 0$ under $H_0^{(\sigma)}$ and under the local alternatives
in order to establish that $\{\psi_n^1\}_{n=1}^\infty$ is AUMPI. Since the statistical model satisfies local asymptotic normality, by Le Cam's first lemma it is enough to show that $\mathrm{LR}_1 - T_n^1 \to 0$ under $H_0^{(\sigma)}$.
Similarly, it suffices to show that $\mathrm{LR}_2-T_n^2\to 0$ under $H_0^{(\theta)}$
to show that $\{\psi_n^2\}_{n=1}^\infty$ is AUMPI.
For these purposes,
we introduce some auxiliary results related to $\mathbb{H}_n$ and the maximum-likelihood estimators.

\begin{prop}\label{prop:ex-3}
Assume conditions (A1)--(A4) and (A6). Then,
    $\hat{\sigma}_{n}^{*} \overset{P}{\rightarrow} \sigma_{0}$,
and under \(H_0^{(\sigma)}\),
    $\tilde{\sigma}_{n}^{*} \overset{P}{\rightarrow} \sigma_{0}$
    as $n\to\infty$.
\begin{proof}
Following the discussion in Section~3.2 of Ogihara~\cite{ogi23}, for any \(\delta>0\) there exists \(\eta>0\) such that
\[
    \inf_{|\sigma-\sigma_{0}|\geq\delta} \Bigl(-\mathcal{Y}_{1}(\sigma)\Bigr) \geq \eta.
\]
Since, by definition, \(H_{n}(\hat{\sigma}_{n}^{*}, \theta_{0}) - H_{n}(\sigma_{0}, \theta_{0}) \geq 0\), it follows from Lemma \ref{lem:ex-2} that for any \(\epsilon>0\), for sufficiently large \(n\),
\begin{align*}
    P\bigl(\lvert\hat{\sigma}_{n}^{*} - \sigma_{0}\rvert\geq \delta\bigr)
    &\leq P\Biggl(n^{-1}\Bigl(H_{n}(\hat{\sigma}_{n}^{*}, \theta_{0}) - H_{n}(\sigma_{0}, \theta_{0})\Bigr)-\mathcal{Y}_{1}(\hat{\sigma}_{n}^{*})\geq \eta\Biggr)\\[1mm]
    &\leq P\Biggl(\sup_{\lvert\sigma-\sigma_{0}\rvert\geq \delta}\Bigl\{ n^{-1}\Bigl(H_{n}(\sigma, \theta_{0}) - H_{n}(\sigma_{0}, \theta_{0})\Bigr)-\mathcal{Y}_{1}(\sigma)\Bigr\}\geq \eta\Biggr)\\[1mm]
    &< \epsilon.
\end{align*}
Thus, \(\hat{\sigma}_{n}^{*} \overset{P}{\rightarrow} \sigma_{0}\). Under \(H_0^{(\sigma)}\), since \(H_{n}(\tilde{\sigma}_{n}^{*}, \theta_{0}) - H_{n}(\sigma_{0}, \theta_{0}) \geq 0\), an analogous argument shows that \(\tilde{\sigma}_{n}^{*} \overset{P}{\rightarrow} \sigma_{0}\).
\end{proof}
\end{prop}

For \(x,y\in\Theta_{1}\) and \(l\in \{0,1\}\), define
\[
    \tilde{\Gamma}_{1, n}^{l,*}(x, y)=
    -(1+l)\int_{0}^{1} \frac{(1-u)^l}{n}\partial^{2}_{\sigma} \mathbb{H}_{n}\Bigl(y + u(x-y), \theta_{0}\Bigr)\,\mathrm{d}u,
\]
and, similarly to \(\Gamma_{1}\), partition it as
\[
    \tilde{\Gamma}_{1, n}^{l,*}(x, y) =
    \begin{pmatrix*}[c]
        \tilde{\Gamma}_{1,1,n}^{l,*} & \tilde{\Gamma}_{1,2,n}^{l,*} \\[1mm]
        (\tilde{\Gamma}_{1,2,n}^{l,*})^{\top} & \tilde{\Gamma}_{1,3,n}^{l,*}
    \end{pmatrix*}(x, y).
\]

\begin{lem}\label{lem:ex-4}
Assume conditions (A1)--(A4) and (A6). Then, under \(H_{0}^{(\sigma)}\),
\[
    \tilde{\Gamma}_{1,n}^{1,*}(\tilde{\sigma}_n^*, \hat{\sigma}_n^*)  \overset{P}{\rightarrow} \Gamma_{1},\quad 
    \tilde{\Gamma}_{1, n}^{0,*}(\tilde{\sigma}_{n}^{*}, \sigma_{0}) \overset{P}{\rightarrow} \Gamma_{1},\quad 
    \tilde{\Gamma}_{1, n}^{0,*}(\hat{\sigma}_{n}^{*}, \sigma_{0}) \overset{P}{\rightarrow} \Gamma_{1},
\]
as \(n\to\infty\).
\begin{proof}
Using Lemma \ref{lem:ex-2} and Proposition \ref{prop:ex-3}, the result can be shown in a manner analogous to that of Lemmas \ref{lem:1} and \ref{lem:2}.
\end{proof}
\end{lem}

\begin{lem}\label{lem:ex-6}
Assume conditions (A1)--(A4) and (A6). Then, under \(H_{0}^{(\sigma)}\),
\[
    \Bigl|\sqrt{n}\bigl(\tilde{\sigma}_n - \hat{\sigma}_n\bigr) - \sqrt{n}\bigl(\tilde{\sigma}_{n}^{*} - \hat{\sigma}_{n}^{*}\bigr)\Bigr| \overset{P}{\rightarrow} 0.
\]
\begin{proof}
Note that
\begin{align*}
    &\Bigl|\sqrt{n}\bigl(\tilde{\sigma}_n - \hat{\sigma}_n\bigr) - \sqrt{n}\bigl(\tilde{\sigma}_{n}^{*} - \hat{\sigma}_{n}^{*}\bigr)\Bigr|\\[1mm]
    &\quad \leq \Bigl|\sqrt{n}\bigl(\tilde{\sigma}_n - \sigma_{0}\bigr) - \sqrt{n}\bigl(\tilde{\sigma}_{n}^{*} - \sigma_{0}\bigr)\Bigr|
    + \Bigl|\sqrt{n}\bigl(\hat{\sigma}_n - \sigma_{0}\bigr) - \sqrt{n}\bigl(\hat{\sigma}_{n}^{*} - \sigma_{0}\bigr)\Bigr|.
\end{align*}
By Taylor's theorem, we have
\[
    \frac{1}{\sqrt{n}}\partial_{P_{0}\sigma}\mathbb{H}_{n}^{1}(\tilde{\sigma}_{n}) - \frac{1}{\sqrt{n}}\partial_{P_{0}\sigma}\mathbb{H}_{n}^{1}(\sigma_{0}) 
    = -\tilde{\Gamma}_{1,3,n}^0(\tilde{\sigma}_{n}, \sigma_{0})\,\sqrt{n}\,P_{0}\bigl(\tilde{\sigma}_{n} - \sigma_{0}\bigr),
\]
and
\[
    \frac{1}{\sqrt{n}}\partial_{P_{0}\sigma}\mathbb{H}_{n}(\tilde{\sigma}_{n}^{*}, \theta_{0}) - \frac{1}{\sqrt{n}}\partial_{P_{0}\sigma}\mathbb{H}_{n}(\sigma_{0}, \theta_{0}) 
    = -\tilde{\Gamma}_{1,3,n}^{0,*}(\tilde{\sigma}_{n}^{*}, \sigma_{0})\,\sqrt{n}\,P_{0}\bigl(\tilde{\sigma}_{n}^{*} - \sigma_{0}\bigr).
\]
By the definitions of \(\tilde{\sigma}_{n}\) and \(\tilde{\sigma}_{n}^{*}\),
\[
    \frac{1}{\sqrt{n}}\partial_{P_{0}\sigma}\mathbb{H}_{n}^{1}(\sigma_{0}) 
    = \tilde{\Gamma}_{1,3,n}^0(\tilde{\sigma}_{n}, \sigma_{0})\,\sqrt{n}\,P_{0}\bigl(\tilde{\sigma}_{n} - \sigma_{0}\bigr),
\]
\[
    \frac{1}{\sqrt{n}}\partial_{P_{0}\sigma}\mathbb{H}_{n}(\sigma_{0}, \theta_{0}) 
    = \tilde{\Gamma}_{1,3,n}^{0,*}(\tilde{\sigma}_{n}^{*}, \sigma_{0})\,\sqrt{n}\,P_{0}\bigl(\tilde{\sigma}_{n}^{*} - \sigma_{0}\bigr).
\]
Under \(H_{0}^{(\sigma)}\), we have
\[
    \Bigl|\sqrt{n}\bigl(\tilde{\sigma}_{n} - \sigma_{0}\bigr) - \sqrt{n}\,P_{0}\bigl(\tilde{\sigma}_{n} - \sigma_{0}\bigr)\Bigr|
    = \Bigl|\sqrt{n}\,P_{0}\bigl(\tilde{\sigma}_{n} - \sigma_{0}\bigr) - \sqrt{n}\,P_{0}\bigl(\tilde{\sigma}_{n}^{*} - \sigma_{0}\bigr)\Bigr|.
\]
Therefore, by Lemmas \ref{lem:ex-1}, \ref{lem:2}, and \ref{lem:ex-4}, we obtain
\[
    \Bigl|\sqrt{n}\bigl(\tilde{\sigma}_{n} - \sigma_{0}\bigr) - \sqrt{n}\bigl(\tilde{\sigma}_{n}^{*} - \sigma_{0}\bigr)\Bigr| \overset{P}{\rightarrow} 0.
\]
Similarly, since
\[
    \frac{1}{\sqrt{n}}\partial_{\sigma}\mathbb{H}_{n}^{1}(\sigma_{0}) 
    = \tilde{\Gamma}_{1,3,n}^0(\hat{\sigma}_{n}, \sigma_{0})\,\sqrt{n}\bigl(\hat{\sigma}_{n} - \sigma_{0}\bigr),
\]
and
\[
    \frac{1}{\sqrt{n}}\partial_{\sigma}\mathbb{H}_{n}(\sigma_{0}, \theta_{0}) 
    = \tilde{\Gamma}_{1,3,n}^{0,*}(\hat{\sigma}_{n}^{*}, \sigma_{0})\,\sqrt{n}\bigl(\hat{\sigma}_{n}^{*} - \sigma_{0}\bigr),
\]
an analogous argument using Lemmas \ref{lem:ex-1}, \ref{lem:2}, and \ref{lem:ex-4} shows that
\[
    \Bigl|\sqrt{n}\bigl(\hat{\sigma}_{n} - \sigma_{0}\bigr) - \sqrt{n}\bigl(\hat{\sigma}_{n}^{*} - \sigma_{0}\bigr)\Bigr| \overset{P}{\rightarrow} 0.
\]
Thus, the conclusion follows.
\end{proof}
\end{lem}

The following theorem completes the proof that \(\{\psi_n^1\}_{n=1}^\infty\) is AUMPI.
\begin{thm}\label{LR1-thm}
Assume conditions (A1)--(A4) and (A6). Then, under \(H_{0}^{(\sigma)}\),
\[
    \mathrm{LR}_{1} - T_{n}^{1} \overset{P}{\rightarrow} 0.
\]
\begin{proof}
From (\ref{eq:4}) and a similar expression for
$\mathbb{H}_{n}(\tilde{\sigma}_{n}^{*},\theta_{0})-\mathbb{H}_{n}(\hat{\sigma}_{n}^{*},\theta_{0})$,
we have
\begin{equation}\label{eq:ex-diff}
    \begin{split}
        \mathrm{LR}_{1} - T_{n}^{1} 
        &= \tilde{\Gamma}_{1,n}^{1}(\tilde{\sigma}_n,\hat{\sigma}_n) \Bigl\llbracket\Bigl[\sqrt{n}\bigl(\tilde{\sigma}_{n} - \hat{\sigma}_{n}\bigr)\Bigr]^{\otimes 2}\Bigr\rrbracket
        - \tilde{\Gamma}_{1,n}^{1,\ast}(\tilde{\sigma}_n^\ast,\hat{\sigma}_n^\ast) \Bigl\llbracket\Bigl[\sqrt{n}\bigl(\tilde{\sigma}_{n}^{*} - \hat{\sigma}_{n}^{*}\bigr)\Bigr]^{\otimes 2}\Bigr\rrbracket \\[1mm]
        &= \Bigl(\tilde{\Gamma}_{1,n}^{1}(\tilde{\sigma}_n,\hat{\sigma}_n) - \tilde{\Gamma}_{1,n}^{1,\ast}(\tilde{\sigma}_n^\ast,\hat{\sigma}_n^\ast)\Bigr)
        \Bigl\llbracket\Bigl[\sqrt{n}\bigl(\tilde{\sigma}_{n} - \hat{\sigma}_{n}\bigr)\Bigr]^{\otimes 2}\Bigr\rrbracket \\[1mm]
        &\quad + \Bigl\{\sqrt{n}\bigl(\tilde{\sigma}_{n} - \hat{\sigma}_{n}\bigr) - \sqrt{n}\bigl(\tilde{\sigma}_{n}^{*} - \hat{\sigma}_{n}^{*}\bigr)\Bigr\}^{\top} \tilde{\Gamma}_{1,n}^{1,\ast}(\tilde{\sigma}_n^\ast,\hat{\sigma}_n^\ast)
        \Bigl\{\sqrt{n}\bigl(\tilde{\sigma}_{n} - \hat{\sigma}_{n}\bigr) + \sqrt{n}\bigl(\tilde{\sigma}_{n}^{*} - \hat{\sigma}_{n}^{*}\bigr)\Bigr\}.
    \end{split}
\end{equation}
By Lemmas \ref{lem:ex-1} and \ref{lem:ex-4}, we have
\[
    \tilde{\Gamma}_{1,n}^1(\tilde{\sigma}_n,\hat{\sigma}_n) \overset{P}{\rightarrow} \Gamma_{1}, \quad \tilde{\Gamma}_{1,n}^{1,\ast}(\tilde{\sigma}_n^\ast,\hat{\sigma}_n^\ast)\overset{P}{\rightarrow} \Gamma_{1}.
\]
Furthermore, by Lemma \ref{lem:ex-4}, under \(H_{0}^{(\sigma)}\) the sequence \(\sqrt{n}(\tilde{\sigma}_{n} - \hat{\sigma}_{n})\) is \(P\)-tight, so the first term in \eqref{eq:ex-diff} converges in probability to zero. Also, by Lemma \ref{lem:ex-6} the sequence \(\sqrt{n}(\tilde{\sigma}_{n}^{*} - \hat{\sigma}_{n}^{*})\) is \(P\)-tight, and a further application of Lemma \ref{lem:ex-6} shows that the second term in \eqref{eq:ex-diff} also converges in probability to zero. Therefore, the conclusion follows.
\end{proof}
\end{thm}

\subsubsection*{Proof of Theorem~\ref{aumpi-thm}}

From the discussion preceding Proposition~\ref{prop:ex-3} and Theorem~\ref{LR1-thm}, it follows that the sequence $\{\psi_n^1\}_{n=1}^\infty$ is AUMPI.
We obtain the AUMPI property for $\psi_n^2$ by Lemmas~\ref{lem:ex-1} and~\ref{lem:ex-2}, Proposition 3.16 in Ogihara~\cite{ogi23}, and an argument similarly to the above one.

\qed 

\section*{Acknowledgments}
Teppei Ogihara was supported by Japan Society for the Promotion of Science KAKENHI Grant Number JP21H00997.
In this research work, we used the UTokyo Azure. 

\noindent
\verb|https://utelecon.adm.u-tokyo.ac.jp/en/research_computing/utokyo_azure/| 

\begin{small}
\bibliographystyle{abbrv}
\bibliography{wileyNJD-AMA}
\end{small}

\appendix

\section{The test statistic constructed based on the Hayashi--Yoshida estimator\label{app1}}

In Hayashi and Yoshida~\cite{hay-yos05}, the Hayashi--Yoshida estimator for diffusion processes observed nonsynchronously was introduced:
\[
    \mathrm{HY}_{n} = \sum_{i,j} \Delta_i^1 X\, \Delta_j^2 X\, \mathbf{1}_{\{I_{i}^{1} \cap I_{j}^{2} \neq \emptyset\}}.
\]
In Hayashi and Yoshida~\cite{hay-yos08}, the authors showed that, for a fixed observation interval \([0,T]\) (with \(T>0\)) and as \(\max_{i,l}\lvert I_i^l \rvert\to 0\), one has
\begin{equation}
    \label{eq:HY_normality}
    \sqrt{n}\Bigl(\mathrm{HY}_{n} - \langle X^1,X^2\rangle_T\Bigr)
    \overset{d}{\rightarrow} \mathcal{N}(0, c) \quad (\text{as } n\to\infty),
\end{equation}
where \(c\) is a positive constant. Here, for any interval \(I\), we define
\[
    v(I) = \int_{I}[b_{t}b_{t}^\top(\sigma_0)]_{12}\,\mathrm{d}t, \quad
    v^{l}(I) = \int_{I}[b_{t}b_{t}^\top(\sigma_0)]_{ll}\,\mathrm{d}t \quad (l=1,2).
\]
Furthermore, it is assumed that
\[
    n\Biggl\{\sum_{i,j} v^{1}(I_{i}^{1})v^{2}(I_{j}^{2})\,\mathbf{1}_{\{I_{i}^{1}\cap I_{j}^{2}\neq\emptyset\}}
    + \sum_{i} v(I_{i}^{1})^{2} + \sum_{j} v(I_{j}^{2})^{2}
    - \sum_{i,j} v(I_{i}^{1}\cap I_{j}^{2})^{2}\Biggr\} \overset{P}{\rightarrow} c
\]
as $n\to\infty$.

Define
\[
    \hat{\Lambda}_{1,n} = n\frac{\mathrm{HY}_{n}^{2}}{T^{2}}\Biggl(\sum_{i}\lvert I_{i}^{1} \rvert^{2} + \sum_{j}\lvert I_{j}^{2}\rvert^{2}\Biggr),
\]
\[
    \hat{\Lambda}_{2,n} = -n\frac{\mathrm{HY}_{n}^{2}}{T^{2}}\sum_{i,j}\lvert I_{i}^{1}\cap I_{j}^{2}\rvert^{2},
\]
\[
    \hat{\Lambda}_{3,n} = n\sum_{i,j} (\Delta_{i}^{1}X)^{2} (\Delta_{j}^{2}X)^{2}\,\mathbf{1}_{\{I_{i}^{1}\cap I_{j}^{2}\neq \emptyset\}} + 2\tilde{\Lambda}_{2,n}.
\]
If the diffusion coefficient \(b_{t}\) is time-invariant, then a consistent estimator \(U_n\) of the asymptotic variance \(c\) is constructed by
\[
    U_{n} = \hat{\Lambda}_{1,n} + \hat{\Lambda}_{2,n} + \hat{\Lambda}_{3,n} \overset{P}{\rightarrow} c \quad \text{as } n\to\infty.
\]
In the case where \(b_{t}\) depends on \(t\), a similar estimator can be constructed by approximating \(b_{t}\) by a piecewise constant function.

Consider the hypothesis test concerning the covariation \(\langle X^1,X^2\rangle_T\):
\begin{align*}
    H_0&: \langle X^1,X^2\rangle_T = 0, \\
    H_1&: \text{not } H_0.
\end{align*}
Using the Hayashi--Yoshida estimator \(\mathrm{HY}_{n}\) and the estimator \(U_n\), the test statistic
\[
    V_{n} = \frac{\sqrt{n}\mathrm{HY}_{n}}{\sqrt{U_{n}}}
\]
is formed. Under \(H_{0}\) it is shown that
\[
    V_{n}\overset{d}{\rightarrow}\mathcal{N}(0, 1) \quad (\text{as } n\to\infty),
\]
and under \(H_{1}\), for any $M>0$,
\[
    P(\lvert V_{n} \rvert\leq M)\to 0\quad (\text{as } n\to\infty).
\]

\end{document}